\documentclass[12pt]{article}

\usepackage{graphics,amsmath,amssymb,amsthm,mathrsfs}

\newcommand{\br}{\mathbb{R}}

\newcommand{\varep}{\varepsilon}
\newcommand{\brd}{\mathbb{R}^d}
\newcommand{\wD}{\widetilde{D}}
\newcommand{\wN}{\widetilde{N}}

\newtheorem{thm}{Theorem}[section]
\newtheorem{lemma}[thm]{Lemma}

\newtheorem{prop}[thm]{Proposition}

\newtheorem{remark}[thm]{Remark}

\numberwithin{equation}{section}

\begin{document}

\bibliographystyle{amsplain}

\title{Periodic Homogenization \\ of Green and Neumann Functions}

\author{Carlos E. Kenig\thanks{Supported in part by NSF grant DMS-0968472}
 \and Fanghua Lin \thanks{Supported in part by NSF grant DMS-0700517}
\and Zhongwei Shen\thanks{Supported in part by NSF grant DMS-0855294}}

\date{ }

\maketitle

\begin{abstract}
For a family of second-order elliptic operators with rapidly oscillating 
periodic coefficients,  we study the asymptotic behavior of
the Green and Neumann functions,
using Dirichlet and Neumann correctors.
As a result we obtain asymptotic expansions of Poisson kernels and the Dirichlet-to-Neumann
maps as well as near optimal convergence rates in $W^{1,p}$
for solutions with Dirichlet or Neumann boundary conditions.
\end{abstract}

\section{Introduction}

The main purpose of this paper is to study the asymptotic behavior of
the Green and Neumann functions
for a family of elliptic operators
with rapidly oscillating coefficients.
More precisely, consider
\begin{equation}\label{operator}
\mathcal{L}_\varep =-\text{div} \big( A\left({x}/{\varep}\right)\nabla\big)
=-\frac{\partial }{\partial x_i}
\left[ a_{ij}^{\alpha\beta}\left(\frac{x}{\varep}\right)
\frac{\partial}{\partial x_j} \right], \quad \varep>0
\end{equation}
(the summation convention is used throughout the paper).
We will assume that $A(y)=(a_{ij}^{\alpha\beta} (y))$
with $1\le i,j\le d$ and $1\le \alpha,\beta\le m$ is real and satisfies the
ellipticity condition
\begin{equation}\label{ellipticity}
\mu |\xi|^2 \le a_{ij}^{\alpha\beta} (y) \xi_i^\alpha\xi_j^\beta
\le \frac{1}{\mu} |\xi|^2 \ \ 
\text{ for } y\in \br^d \text{ and } \xi=(\xi_i^\alpha)\in \br^{dm},
\end{equation}
where $\mu>0$, and the periodicity condition
\begin{equation}\label{periodicity}
A(y+z)=A(y) \quad \text{ for } y\in \br^d \text{ and } z\in \mathbb{Z}^d.
\end{equation}
We will also impose the smoothness condition
\begin{equation}\label{smoothness}
|A(x)-A(y)|\le \tau |x-y|^\lambda
\quad \text{ for some }\lambda\in (0,1] \text{ and } \tau\ge 0.
\end{equation}
Let $G_\varep (x,y)=\big(G_\varep^{\alpha\beta}(x,y)\big)$ and
$N_\varep (x,y)=\big(N_\varep^{\alpha\beta}(x,y) \big)$
denote the Green and Neumann functions respectively, for $\mathcal{L}_\varep$
in a bounded domain $\Omega$, with pole at $y$.
We are interested in the asymptotic behavior, as $\varep\to 0$, of
$G_\varep(x,y)$, $N_\varep(x,y)$, $\nabla_x G_\varep (x,y)$ and
$\nabla_x N_\varep (x,y)$ as well as $\nabla_x\nabla_y G_\varep (x,y)$
an $\nabla_x\nabla_y N_\varep (x,y)$.
We shall use $G_0(x,y)$ and $N_0(x,y)$ to denote the Green and
Neumann functions respectively, for the homogenized (effective)
operator $\mathcal{L}_0$ in $\Omega$.

Let $P_j^\beta =x_j (0, \dots, 1, \dots)$
with $1$ in the $\beta^{th}$ position for $1\le j\le d$ and $1\le \beta\le m$.
To state our main results, we need to introduce the matrix of
Dirichlet correctors $\Phi_{\varep, j}^\beta
=\big(\Phi_{\varep, j}^{1\beta}, \dots, \Phi_{\varep, j}^{m\beta}\big)$ in $\Omega$,
defined by
\begin{equation}\label{Dirichlet-corrector}
\mathcal{L}_\varep (\Phi_{\varep, j}^\beta) =0
\quad \text{ in } \Omega
\quad \text{ and } \quad \Phi_{\varep,j}^\beta=P_j^\beta \quad \text{ on } \partial\Omega,
\end{equation}
as well as the matrix of Neumann correctors 
$\Psi_{\varep, j}^\beta
=\big(\Psi_{\varep, j}^{1\beta}, \dots, \Psi_{\varep, j}^{m\beta}\big)$ in $\Omega$, defined by
\begin{equation}\label{Neumann-corrector}
\mathcal{L}_\varep (\Psi_{\varep, j}^\beta) =0
\quad \text{ in } \Omega
\quad \text{ and } \quad \frac{\partial}{\partial\nu_\varep}
\big(\Psi_{\varep,j}^\beta\big)=
\frac{\partial}{\partial\nu_0}
\big(P_j^\beta\big) \quad \text{ on } \partial\Omega.
\end{equation}
Here $\partial/\partial\nu_\varep$ denotes the conormal derivative associated with
$\mathcal{L}_\varep$ for $\varep\ge 0$.

The following are the main results of the paper.

\begin{thm}\label{theorem-A}
Let $\mathcal{L}_\varep =-\text{\rm div}\big(A(x/\varep)\nabla \big)$
with the matrix $A(y)$ satisfying conditions (\ref{ellipticity}),
(\ref{periodicity}) and (\ref{smoothness}).
Then for any $x,y\in \Omega$,
\begin{equation}\label{Green's-size}
|G_\varep (x,y)-G_0 (x,y)|\le \frac{C\varep}{|x-y|^{d-1}}
\end{equation}
if $\Omega$ is a bounded $C^{1,1}$ domain, and
\begin{equation}\label{Green's-derivative}
\big|\frac{\partial}{\partial x_i}
\big\{ G_\varep^{\alpha\beta} (x,y)\big\}
-\frac{\partial}{\partial x_i} \big\{ \Phi_{\varep, j}^{\alpha\gamma} (x)\big\}
\cdot \frac{\partial}{\partial x_j}
\big\{ G_0^{\gamma\beta}(x,y)\big\} \big|
\le \frac{ C\varep \ln \left[\varep^{-1} |x-y|+2\right]}{|x-y|^d}
\end{equation}
if $\Omega$ is a bounded $C^{2,\eta}$ domain
for some $\eta\in (0,1)$,
where $C$ depends only on $d$, $m$, $\mu$, $\tau$, $\lambda$ and $\Omega$.
\end{thm}

\begin{thm}\label{theorem-B}
Suppose that $A(y)$ satisfies the same conditions as in Theorem \ref{theorem-A}.
Also assume that $A^*=A$, i.e., $a_{ij}^{\alpha\beta}(y)
=a_{ji}^{\beta\alpha}(y)$ for $1\le i,j\le d$ and $1\le \alpha,\beta\le m$.
Then for any $x,y\in \Omega$,
\begin{equation}\label{Neumann's-size}
|N_\varep (x,y)-N_0 (x,y)|\le \frac{C\varep\ln \left[ \varep^{-1}
|x-y| +2\right]}{|x-y|^{d-1}}
\end{equation}
if $\Omega$ is a bounded $C^{1,1}$ domain,
where $C$ depends only on $d$, $m$, $\mu$, $\tau$, $\lambda$ and $\Omega$.
Moreover, if $\Omega$ is a bounded $C^{2, \eta}$ domain for some $\eta\in (0,1)$,
\begin{equation}\label{Neumann-derivative}
\big|\frac{\partial}{\partial x_i}
\big\{ N_\varep^{\alpha\beta} (x,y)\big\}
-\frac{\partial}{\partial x_i} \big\{ \Psi_{\varep, j}^{\alpha\gamma} (x)\big\}
\cdot \frac{\partial}{\partial x_j}
\big\{ N_0^{\gamma\beta}(x,y)\big\} \big|
\le \frac{ C_t\, \varep^t \ln \left[ \varep^{-1} M +2\right]}{|x-y|^{d-1+t}}
\end{equation}
for any $x,y\in \Omega$ and $t\in (0,1)$,
where $M=\text{\rm diam}(\Omega)$
and $C_t$ depends only on $d$, $m$, $\mu$, $\tau$, $\lambda$, $t$ and $\Omega$.
\end{thm}

A few remarks are in order.

\begin{remark}\label{remark-1.1}
{\rm
In the case of a scalar equation $(m=1)$, the estimate (\ref{Green's-size})
holds for bounded measurable coefficients satisfying (\ref{ellipticity})
and (\ref{periodicity}) (see Theorem \ref{theorem-3.1.1}).
}
\end{remark}

\begin{remark}\label{remark-1.2}
{\rm 
The matrix of Dirichlet correctors $\big(\Phi_{\varep, j}^\beta\big)$
was introduced in \cite{AL-1987} to establish the boundary
Lipschitz estimates for solutions with Dirichlet conditions,
while the matrix of Neumann correctors $\big(\Psi_{\varep, \beta}^\beta\big)$
was introduced in \cite{KLS1} to establish the same estimates
for solutions with Neumann boundary conditions.
It is known that $\|\Phi_{\varep, j}^\beta -P_j^\beta\|_{L^\infty(\Omega)}\le C \varep$
and $\|\nabla \Phi_{\varep, j}^\beta\|_{L^\infty(\Omega)}
+\|\nabla \Psi_{\varep, j}^\beta\|_{L^\infty(\Omega)}\le C$.
Under the condition $\Psi_{\varep, j}^\beta (x_0)=P_j^\beta (x_0)$
for some $x_0\in \Omega$, we also have
$\|\Psi_{\varep, j}^\beta -P_j^\beta\|_{L^\infty(\Omega)}
\le C \varep \ln [\varep^{-1}M+2]$ 
(see Propositions \ref{prop-2.2} and \ref{prop-2.3}).
}
\end{remark}

\begin{remark}\label{remark-1.3}
{\rm
Estimates (\ref{Green's-size}) and (\ref{Neumann's-size})
in Theorems \ref{theorem-A} and \ref{theorem-B}
allow us to establish $O(\varep)$ estimates for
$\| u_\varep-u_0\|_{L^p(\Omega)}$ ($1<p\le \infty$)
for solutions with Dirichlet or Neumann boundary conditions
(see Theorems \ref{theorem-3.1.2} and \ref{theorem-4.1.3}).
More importantly, estimates (\ref{Green's-derivative}) and
(\ref{Neumann-derivative}) yield near optimal convergence rates
in $W^{1,p}(\Omega)$ for any $1<p<\infty$. In fact,
let $\mathcal{L}_\varep (u_\varep)=F$ in $\Omega$
and $u_\varep=0$ on $\partial\Omega$.
Then
\begin{equation}\label{W-1-p-estimate-1}
\| u_\varep -u_0 -\big\{ \Phi_{\varep, j}^\beta
-P_j^\beta\big\} \frac{\partial u_0^\beta}{\partial x_j}\|_{W_0^{1,p}(\Omega)}
\le C_p \, \varep \big\{ \ln [\varep^{-1}M +2]\big\}^{4|\frac12-\frac{1}{p}|}
 \| F\|_{L^p(\Omega)}.
\end{equation}
In the case of Neumann boundary conditions
we obtain
\begin{equation}\label{W-1-p-estimate-2}
\| u_\varep -u_0 -\big\{ \Psi_{\varep, j}^\beta
-P_j^\beta\big\} \frac{\partial u_0^\beta}{\partial x_j}\|_{W^{1,p}(\Omega)}
\le C_{t,p}\, \varep^t \| F\|_{L^p(\Omega)}
\end{equation}
for any $t\in (0,1)$, where 
$\mathcal{L}_\varep (u_\varep)=F$ in $\Omega$, 
$\frac{\partial u_\varep}{\partial\nu_\varep}=0$ on $\partial\Omega$
and $\int_\Omega F=\int_{\partial\Omega} u_\varep =0$.
See subsections 3.2 and 4.2 for details.
Let
$w_\varep =u_\varep -u_0 -\varep \chi_j^\beta (x/\varep) \frac{\partial u_0^\beta}
{\partial x_j}$, where
$(\chi_j^\beta (y))$
denotes the matrix of correctors for $\mathcal{L}_1$ in $\br^d$.
Estimates (\ref{W-1-p-estimate-1}) and (\ref{W-1-p-estimate-2}) should be compared
 to the well known $O(\varep^{1/2})$ estimate:
$
\| w_\varep \|_{H^1(\Omega)}=O(\varep^{1/2})
$
(see e.g. \cite{bensoussan-1978}), 
and to the following estimate,
\begin{equation}\label{H-0.5-estimate}
\| w_\varep \|_{H^{1/2}(\Omega)}
+\left\{ \int_\Omega |\nabla w_\varep (x)|^2 \, \text{dist}(x, \partial\Omega)
\, dx \right\}^{1/2}
\le C \varep \| F\|_{L^2(\Omega)},
\end{equation}
proved in \cite[Theorems 3.4 and 5.2]{KLS2}.
Due to the presence of a boundary layer, 
both the $O(\varep^{1/2})$ estimate 
and (\ref{H-0.5-estimate}) are more or less sharp.
The Dirichlet and Neumann correctors are introduced precisely to deal with
boundary layer phenomena in periodic homogenization.
}
\end{remark}

\begin{remark}\label{remark-1.4}
{\rm
Our approach to Theorems \ref{theorem-A}
and \ref{theorem-B} also leads to asymptotic estimates of $\nabla_x\nabla_y G_\varep (x,y)$
and $\nabla_x\nabla_y N_\varep (x,y)$
(see subsection 3.3 and Remark \ref{remark-4.2.1}).
As a result we obtain asymptotic expansions for
$(\partial/\partial x_i) (\mathcal{L}_\varep)^{-1} 
(\partial/\partial x_j)$ and $\Lambda_\varep$, the Dirichlet-to-Neumann map 
associated with $\mathcal{L}_\varep$.
}
\end{remark}

The asymptotic expansion of the fundamental solutions as well as the heat 
kernels for $\mathcal{L}_\varep$ 
in $\br^d$ has been studied, using the method of Bloch waves;
see e.g. \cite{Sevostjanova-1982,ERS-2001} and their references
(also see \cite{Kozlov-1982} for results obtained by the method of
$G$-convergence). 
In the presence of boundary, the Bloch representation is 
no longer available.
In a series of papers \cite{AL-1987, AL-1989-ho, AL-1989-II},
M. Avellaneda and F. Lin introduced the compactness methods,
which originated in the regularity theory in the calculus of variations
and minimal surfaces,
to homogenization problems.
In particular, they established an asymptotic expansion for
$\nabla_y G_\varep (x,y)$ in \cite{AL-1989-ho}, using Dirichlet correctors.
As a result, it was proved in \cite{AL-1989-ho} that
if $\Omega$ is $C^{1,\eta}$ for some $\eta\in (0,1)$,
\begin{equation}\label{1.10}
P_\varep (x,y)=P_0(x,y)\omega_\varep (y) +R_\varep(x,y) \qquad
\text{ for } x\in \Omega \text{ and } y\in \partial\Omega,
\end{equation}
where $P_\varep (x,y)$ ($\varep\ge 0$) denotes the Poisson 
kernel for $\mathcal{L}_\varep$ in $\Omega$,
\begin{equation}\label{1.11}
\omega_\varep (y)
=\frac{\partial}{\partial n(y)}
\big\{ \Phi_{\varep, k}^* (y)\big\} \cdot n_k(y) h(y) \cdot
n_i(y)n_j(y) a_{ij}(y/\varep)
\end{equation}
and the remainder $R_\varep (x,y)$ satisfies
\begin{equation}\label{1.12}
\lim_{\varep\to  0} \sup\big\{ |R_\varep(x,y)|:
\ x\in E \text{ and } y\in\partial\Omega\big\}
=0
\end{equation}
for any compact subset $E$ of $\Omega$
(the results are stated for the case $m=1$; however, the argument in \cite{AL-1989-ho}
works equally well for elliptic systems).
In (\ref{1.11}) we have used $\Phi_{\varep, k}^*$ to denote the Dirichlet correctors
for $\mathcal{L}^*_\varep$, the adjoint of $\mathcal{L}_\varep$. Also,
$h(y)= (\hat{a}_{ij}n_i (y)n_j(y))^{-1}$ and $(\hat{a}_{ij})$
is the (constant) coefficient matrix of $\mathcal{L}_0$.
The expansion (\ref{1.10}) was used in \cite{AL-1989-ho}
to identify the limit, as $\varep\to 0$, of solutions to a problem
of exact boundary controllability for the wave operator
$\frac{\partial^2}{\partial t^2} +\mathcal{L}_\varep$.

Our Theorem \ref{theorem-A} gives a much more refined estimate of $R_\varep (x,y)$
in (\ref{1.10}) (under the stronger condition $\partial\Omega
\in C^{2,\eta})$. Indeed, it follows from the estimate (\ref{Green's-derivative})
that 
\begin{equation}\label{1.13}
|R_\varep(x,y)|\le \frac{C\, \varep \ln [\varep^{-1}|x-y|+2]}
{|x-y|^d}
\qquad \text{ for any } x\in \Omega, y\in \partial\Omega.
\end{equation}
Besides its applications to boundary control problems,
estimate (\ref{1.13}) may also be used to investigate the Dirichlet problem
\begin{equation}\label{1.14}
\mathcal{L}_\varep (u_\varep) =0 \quad \text{ in }\Omega
\quad \text{ and } \quad
u_\varep (x)=f (x,x/\varep) \quad \text{ on }\partial\Omega,
\end{equation}
where $f(x,y)$ is 1-periodic in $y$. The Dirichlet problem (\ref{1.14})
arises natually in the study of boundary layer phenomena and higher-order
convergence in periodic homogenization (see e.g. 
\cite{Moskow-Vogelius-1,Allaire-Amar-1999,Masmoudi-2008, Masmoudi-2010} and their references).
Let $v_\varep$ be the solution to
\begin{equation}\label{1.15}
\mathcal{L}_0 (v_\varep) =0 \quad \text{ in } \Omega
\quad \text{ and } \quad
v_\varep = f(x, x/\varep) \omega_\varep (x) 
\quad \text{ on } \partial\Omega,
\end{equation}
where $\omega_\varep$ is given by (\ref{1.11}).
It follows from the estimate (\ref{1.13}) that
$$
\|u_\varep -v_\varep \|_{L^p(\Omega)} =O\big( (\varep [\ln (1/\varep)]^2)^{1/p}\big)
\qquad \text{ for any }
1\le p<\infty
$$ 
(see Theorem \ref{approximation-thm}).
This effectively reduces the asymptotic problem (\ref{1.14})
to the study of convergence properties of $\omega_\varep$ on $\partial\Omega$,
under various geometric conditions on $\Omega$.
This line of research will be developed in a future work.

We now describe the main ideas in the proof of Theorems \ref{theorem-A}
and \ref{theorem-B}.
The basic tools in our approach are representation formulas by
Green and Neumann functions, uniform estimates
for Green functions established in \cite{AL-1987},
\begin{equation}\label{1.17}
|\nabla_x G_\varep (x,y)|+|\nabla_y G_\varep (x,y)|\le C |x-y|^{1-d} \quad \text{ and} \quad
|\nabla_x\nabla_y G_\varep (x,y)|\le C|x-y|^{-d},
\end{equation}
and the same estimates obtained in \cite{KLS2} 
for Neumann  functions $N_\varep (x,y)$.
Let $D_r=D(x_0,r)=B(x_0, r)\cap \Omega$ and $\Delta_r
=\Delta(x_0, r)=B(x_0,r)\cap \partial\Omega$ for
some $x_0\in \overline{\Omega}$ and $0<r<r_0$.
First, to establish (\ref{Green's-size}), we will show that if $p>d$,
\begin{equation}\label{1.16}
\| u_\varep -u_0\|_{L^\infty(D_1)}
\le C\int_{D_4} |u_\varep-u_0|\, dx
+ C\varep \|\nabla u_0\|_{L^\infty(D_4)}
+ C_p\, \varep \|\nabla^2 u_0\|_{L^p(D_4)},
\end{equation}
 where $\mathcal{L}_\varep (u_\varep)=\mathcal{L}_0(u_0)$
in $D_4$ and $u_\varep =u_0$ on $\Delta_4$.
This is done by considering
$w_\varep (x)=u_\varep (x)-u_0 (x) -\varep\chi (x/\varep)\nabla u_0$
and using the Green representation formula and
the observation that $\mathcal{L}_\varep (w_\varep)=
\varep \frac{\partial}{\partial x_i}\left(b_{ijk}(x/\varep) \frac{\partial^2 u_0}
{\partial x_j\partial x_k}\right)$, where $b_{ijk}(y)$ is a bounded
periodic function.
Estimate (\ref{Green's-size}) follows from (\ref{1.16})
by a more or less standard argument (see subsection 3.1).
Next, we show in subsection 3.2 that
\begin{equation}\label{1.27}
\aligned
\|\nabla u_\varep  - (\nabla\Phi_\varep ) & (\nabla u_0)\|_{L^\infty(D_r)}
\le  \frac{C}{r^{d+1}} \int_{D_{4r}} |u_\varep -u_0|\, dx
+ C \varep r^{-1}\|\nabla u_0\|_{L^\infty(D_{4r})}\\
& \quad
+C \varep \ln [\varep^{-1} r +2] \| \nabla^2 u_0\|_{L^\infty (D_{4r})}
+ C\varep r^\eta \|\nabla^2 u_0\|_{C^{0, \eta} (D_{4r})},
\endaligned
\end{equation}
if $\mathcal{L}_\varep (u_\varep) =\mathcal{L}_0(u_0)$
in $D_{4r}$ and $u_\varep=u_0$ on $\Delta_{4r}$.
Estimate (\ref{Green's-derivative}) follows easily from
(\ref{1.27}) by taking $u_\varep (x)=G_\varep (x,y_0)$ and
$u_0(x)=G_0(x,y_0)$.
By repeating the argument, estimate (\ref{1.27}) also gives
an asymptotic expansion for $\nabla_x\nabla_y G_\varep(x,y)$ 
(see Theorem \ref{theorem-3.3.1}).
To prove (\ref{1.27}), we let 
\begin{equation}\label{1.18}
w_\varep =u_\varep (x) -u_0 (x) - \big\{ \Phi_{\varep, j}^\beta
-P_j^\beta\big\} \frac{\partial u_0^\beta}{\partial x_j}
\end{equation}
and represent $w_\varep$ in $D_{r}$, using the Green function
in $\wD$, where $\wD$ is a $C^{2,\eta}$ domain such that
$D_{3r}\subset \wD\subset D_{4r}$. 

Although a bit more complicated,
the proof of Theorem \ref{theorem-B} follows the same line of argument as
 Theorem \ref{theorem-A}.
In subsections 4.1 and 4.2 we establish boundary $L^\infty$ and 
Lipschitz estimates similar to (\ref{1.16}) and (\ref{1.27})
for $u_\varep$ and $u_0$ satisfying
$\mathcal{L}_\varep (u_\varep) =\mathcal{L}_0(u_0)$ in $D_{4r}$
and $\frac{\partial u_\varep}{\partial\nu_\varep}
=\frac{\partial u_0}{\partial\nu_0}$ on $\Delta_{4r}$.
The results rely on the uniform $L^p$ and Neumann function
estimates obtained in \cite{Kenig-Shen-2,KLS1}
under the additional symmetry condition $A^*=A$. 

The rest of the paper is organized as follows.
Section 2 contains some basic formulas and estimates which are more or less
known. The case of Dirichlet boundary conditions is treated in Section 3,
while Section 4 is devoted to the case of Neumann boundary conditions.
In Section 5 we prove two inequalities, which are used in subsection 4.3 and are
of interest in their own right,
for the Dirichlet-to-Neumann map $\Lambda_0$.

\section{Preliminaries}

Let $\mathcal{L}_\varep=-\text{\rm div} (A(x/\varep)\nabla )$ 
with $A(y)=\big(a_{ij}^{\alpha\beta}(y)\big)$ satisfying
(\ref{ellipticity})-(\ref{periodicity}).
Let $\chi(y)=\big(\chi_j^{\alpha\beta} (y)\big)$ denote the matrix of correctors 
for $\mathcal{L}_1$ in $\br^d$, where
 $\chi_j^\beta (y)
=(\chi_j^{1\beta}(y), \dots, \chi_j^{m\beta}(y))$ 
is defined by the following cell problem:
\begin{equation}\label{cell-problem}
\left\{
\aligned
& \mathcal{L}_1 (\chi_j^\beta)=-\mathcal{L}_1(P_j^\beta)\quad \text{ in } \br^d,\\
&\chi_j^{\beta} \text{ is periodic with respect to }\mathbb{Z}^d
\text{ and } \int_Y
\chi_j^{\beta} \, dy =0, 
\endaligned
\right.
\end{equation}
for each $1\le j\le d$ and $1\le \beta\le m$.
Here $Y=[0,1)^d\simeq \brd/\mathbb{Z}^d$ and $P_j^\beta (y)
=y_j (0,\dots, 1, \dots, 0)$ with $1$ in the $\beta^{th}$ position.
The homogenized operator is given by
 $\mathcal{L}_0=-\text{div}(\hat{A}\nabla)$, where 
$\hat{A} =(\hat{a}_{ij}^{\alpha\beta})$ and
\begin{equation}
\label{homogenized-coefficient}
\hat{a}_{ij}^{\alpha\beta}
=\int_Y
\left[ a_{ij}^{\alpha\beta}
+a_{ik}^{\alpha\gamma}
\frac{\partial}{\partial y_k}\left( \chi_j^{\gamma\beta}\right)\right]
\, dy.
\end{equation}
It is known that the constant matrix $\hat{A}$ is positive definite with an ellipticity constant
depending only on $d$, $m$ and $\mu$
(see \cite{bensoussan-1978}).

Let
\begin{equation}\label{definition-of-B}
b_{ij}^{\alpha\beta} (y)
=\hat{a}_{ij}^{\alpha\beta}
-a_{ij}^{\alpha\beta} (y)
-a_{ik}^{\alpha\gamma} (y) \frac{\partial}{\partial y_k} \big(\chi_j^{\gamma\beta}\big).
\end{equation}
Since $\int_Y b_{ij}^{\alpha\beta} dy=0$ and $\frac{\partial}{\partial y_i} \big(
b_{ij}^{\alpha\beta}\big)=0$ by (\ref{homogenized-coefficient}) and (\ref{cell-problem}), 
there exists
$F_{kij}^{\alpha\beta}\in H^1(Y)$ such that
\begin{equation}\label{definition-of-F}
b_{ij}^{\alpha\beta} =\frac{\partial}{\partial y_k}
\big\{ F_{kij}^{\alpha\beta}\big\}
\quad \text{ and } \quad
F_{kij}^{\alpha\beta}=-F_{ikj}^{\alpha\beta}.
\end{equation}

\begin{remark}\label{remark-2.1}
{\rm
To see (\ref{definition-of-F}), 
one solves $\Delta f_{ij}^{\alpha\beta}=b_{ij}^{\alpha\beta}$ in $Y$
with $f_{ij}^{\alpha\beta}\in H^1(Y)$ and $\int_Y f_{ij}^{\alpha\beta} dy=0$, 
and let
$$
F_{kij}^{\alpha\beta} (y) 
=\frac{\partial f_{ij}^{\alpha\beta}}{\partial y_k} -\frac{\partial f_{kj}^{\alpha\beta}}
{\partial y_i}
$$
(see e.g. \cite{KLS2}).
Note that if $A(y)$ is H\"older continuous, then $\nabla \chi$ and hence $b_{ij}^{\alpha\beta}$
are H\"older continuous. It follows that $\nabla F$ is H\"older continuous.
In particular,
 $\|\chi\|_{C^1(Y)} +\| F\|_{C^1(Y)}$ is bounded by a constant
depending only on $d$, $m$, $\mu$, $\lambda$ and $\tau$.
In the case of the scalar equation $(m=1)$ with bounded measurable coefficients,
the corrector $\chi$ is H\"older continuous by the De Giorgi -Nash estimates.
This, together with Cacciopoli's inequality and H\"older's inequality, implies that there exist $t>0$ and
$C>0$, depending only on $d$ and $\mu$, such that
 $$
 \int_{B(x,r)} |\nabla\chi|\, dy \le C\, r^{d-1+t}
 \quad \text{ for } x\in Y \text{ and } 0<r<1.
 $$
 In view of (\ref{definition-of-B}) we obtain
 \begin{equation}\label{estimate-of-b}
 \int_{B(x,r)} |b_{ij}(y)|\, dy \le C\, r^{d-1+t} \quad \text{ for } x\in Y \text{ and } 0<r<1.
 \end{equation}
 Since $\Delta f_{ij}=b_{ij}$ in $Y$ and $\int_Y f_{ij} dy=0$,
 \begin{equation}\label{estimate-of-b-2}
 \|\nabla f_{ij}\|_{L^\infty(Y)}
 \le C \| \nabla f_{ij}\|_{L^2(Y)}
 +C \sup_{x\in Y} \int_Y \frac{|b_{ij}(y)|}{|x-y|^{d-1}} \, dy
 \le C(d, \mu),
 \end{equation}
 where we have used (\ref{estimate-of-b}) to estimate the last integral in (\ref{estimate-of-b-2}).
  It follows that  $\|F_{kij}\|_\infty
\le C(d, \mu)$.
}
\end{remark}

The following proposition plays an important role in this paper.
We mention that formula (\ref{formula-2.1}) with  $V_{\varep, j}^\beta (x) 
=P_j^\beta (x)+\varep \chi_j^\beta (x/\varep)$ is known
and may be used to show that $\| u_\varep-u_0\|_{L^2(\Omega)} \le C\varep \| u_0\|_{H^2(\Omega)}$,
where $\mathcal{L}_\varep (u_\varep)=\mathcal{L}_0 (u_0)$ in $\Omega$
and $u_\varep=u_0$ on $\partial\Omega$
(see e.g. \cite{KLS2}).
The proof of our main results on the first-order derivatives of
Green and Neumann functions will rely on (\ref{formula-2.1})
with the matrices of Dirichlet and Neumann correctors respectively in the place
of the functions $V_{\varep, j}^\beta$.

\begin{prop}\label{prop-2.1}
Suppose that $u_\varep\in H^1(\Omega), u_0\in H^2(\Omega)$ and $\mathcal{L}_\varep (u_\varep)
=\mathcal{L}_0(u_0)$ in $\Omega$.
Let
\begin{equation}\label{definition-of-w}
w_\varep (x) =u_\varep (x)-u_0 (x) -\big\{ V_{\varep, j}^\beta (x) -P_j^\beta (x)\big\}
\cdot \frac{\partial u_0^\beta}{\partial x_j},
\end{equation}
where $V_{\varep, j}^\beta= (V_{\varep,j}^{1\beta}, \dots,
V_{\varep, j}^{m\beta})
 \in H^1(\Omega)$ and $\mathcal{L}_\varep (V_{\varep, j}^\beta)
=0$ in $\Omega$ for each $1\le j\le d$ and $1\le \beta\le m$.
Then
\begin{equation}\label{formula-2.1}
\aligned
\left(\mathcal{L}_\varep (w_\varep)\right)^\alpha
=& \varep \frac{\partial}{\partial x_i}
\left\{ \left[ F_{jik}^{\alpha\gamma} \left({x}/{\varep}\right)
\right]
\frac{\partial^2 u_0^\gamma}{\partial x_j\partial x_k}\right\}\\
&+\frac{\partial}{\partial x_i}
\left\{ a_{ij}^{\alpha\beta}\left({x}/{\varep}\right)
\left[ V_{\varep, k}^{\beta\gamma}(x) 
-x_k \delta^{\beta\gamma}\right]
\frac{\partial^2 u_0^\gamma}{\partial x_j\partial x_k}\right\}\\
&
+a_{ij}^{\alpha\beta} \left({x}/{\varep}\right)
\frac{\partial}{\partial x_j}
\left[ V_{\varep, k}^{\beta\gamma}(x)
-x_k \delta^{\beta\gamma}
-\varep \chi_k^{\beta\gamma}\left({x}/{\varep}\right) \right]
\frac{\partial^2 u_0^\gamma}{\partial x_i\partial x_k},
\endaligned
\end{equation}
where $\delta^{\beta\gamma}=1$ if $\beta=\gamma$, and zero otherwise.
\end{prop}

\begin{proof}
Note that
$$
\aligned
a_{ij}^{\alpha\beta}({x}/{\varep})
 \frac{\partial w^\beta_\varep}{\partial x_j}
=a_{ij}^{\alpha\beta} \left(\frac{x}{\varep}\right)
\frac{\partial u^\beta_\varep}{\partial x_j}
-& a_{ij}^{\alpha\beta} \left(\frac{x}{\varep}\right)
\frac{\partial u^\beta_0}{\partial x_j}
-a_{ij}^{\alpha\beta} \left(\frac{x}{\varep}\right)
\frac{\partial}{\partial x_j}
\left\{ V_{\varep, k}^{\beta\gamma}-x_k\delta^{\beta\gamma}
\right\}
\cdot \frac{\partial u_0^\gamma}{\partial x_k}\\
&-a_{ij}^{\alpha\beta}\left(\frac{x}{\varep}\right) 
\left\{ V_{\varep, k}^{\beta\gamma}-x_k\delta^{\beta\gamma}\right\}
\cdot \frac{\partial^2 u^\gamma_0}{\partial x_k
\partial x_j}.
\endaligned
$$
This, together with $\mathcal{L}_\varep (u_\varep)=\mathcal{L}_0( u_0)$, gives
$$
\aligned
\big\{\mathcal{L}_\varep (w_\varep)\big\}^\alpha
=& -\frac{\partial}{\partial x_i} \left\{ 
\left[ \hat{a}_{ij}^{\alpha\beta} -a_{ij}^{\alpha\beta} \left({x}/{\varep}\right)\right]
\frac{\partial u_0^\beta}{\partial x_j}\right\}
+\big\{ \mathcal{L}_\varep (V_{\varep,k}^\gamma -P_k^\gamma) \big\}^\alpha 
\cdot\frac{\partial u_0^\gamma}
{\partial x_k}\\
&+{a}_{ij}^{\alpha\beta} \left( {x}/{\varep}\right) 
\frac{\partial}{\partial x_j} \left\{ V_{\varep, k}^{\beta\gamma}
-x_k \delta^{\beta\gamma} \right\} \cdot \frac{\partial^2 u_0^\gamma}
{\partial x_i \partial x_k}\\
& +\frac{\partial}{\partial x_i}
\left\{a_{ij}^{\alpha\beta}\left({x}/{\varep}\right) 
\left[ V_{\varep, k}^{\beta\gamma}-x_k\delta^{\beta\gamma}\right]
\cdot \frac{\partial^2 u^\gamma_0}{\partial x_k
\partial x_j}\right\}.
\endaligned
$$
Since
$$
\mathcal{L}_\varep \big( V_{\varep, k}^\gamma -P_k^\gamma\big)
=-\mathcal{L}_\varep \big(P_k^\gamma\big)
=\mathcal{L}_\varep \big\{ \varep \chi_k^\gamma ({x}/{\varep})\big\},
$$
we obtain
\begin{equation}\label{2.1.6}
\aligned
\big\{\mathcal{L}_\varep (w_\varep)\big\}^\alpha
= &-\frac{\partial}{\partial x_i} \left\{ b_{ij}^{\alpha\beta} \left(
{x}/{\varep}\right) \frac{\partial u_0^\beta}{\partial x_j}\right\}\\
& +a_{ij}^{\alpha\beta} \left({x}/{\varep}\right)
\frac{\partial}{\partial x_j} \left\{ V_{\varep, k}^{\beta\gamma} (x)
-x_k\delta^{\beta\gamma} -\frac{\partial \chi_k^{\beta\gamma}}
{\partial x_j} \left({x}/{\varep}\right)\right\}
 \cdot \frac{\partial^2 u_0^\gamma}
{\partial x_i \partial x_k}\\
& 
+\frac{\partial}{\partial x_i}
\left\{a_{ij}^{\alpha\beta}\left({x}/{\varep}\right) 
\left[ V_{\varep, k}^{\beta\gamma}-x_k\delta^{\beta\gamma}\right]
\cdot \frac{\partial^2 u^\gamma_0}{\partial x_k
\partial x_j}\right\},
\endaligned
\end{equation}
where $b_{ij}^{\alpha\beta} (y)$ is defined by (\ref{definition-of-B}).
In view of (\ref{definition-of-F}), we may re-write the first term in the right hand side of
(\ref{2.1.6}) as
$$
-\frac{\partial}{\partial x_i}
\left\{ \frac{\partial}{\partial x_k}
\left[ \varep F_{kij}^{\alpha\beta} \left({x}/{\varep}\right)\right] \cdot
\frac{\partial u^\beta_0}{\partial x_j} \right\}
=\varep \frac{\partial}{\partial x_i} \left\{ F_{kij}^{\alpha\beta} 
\left({x}/{\varep}\right) \cdot \frac{\partial^2 u_0^\beta}{\partial x_k\partial x_j}\right\}.
$$
The formula (\ref{formula-2.1}) now follows.
\end{proof}

The next proposition will be used to handle the Neumann boundary condition
(cf. \cite{KLS1}).

\begin{prop}\label{prop-2.4}
Let $w_\varep$ be given by (\ref{definition-of-w}).
Suppose that $\frac{\partial}{\partial \nu_\varep}
\big\{ V_{\varep, k}^\beta\big\} =\frac{\partial}{\partial\nu_0} \big\{P_k^\beta\big\}$.
Then
\begin{equation}\label{conormal-formula}
\left(\frac{\partial w_\varep}{\partial \nu_\varep}\right)^\alpha
=\left(\frac{\partial u_\varep}{\partial \nu_\varep}\right)^\alpha
-\left(\frac{\partial u_0}{\partial \nu_0}\right)^\alpha
-n_i a_{ij}^{\alpha\beta}(x/\varep) \big\{
V_{\varep,k}^{\beta\gamma} -P_k^{\beta\gamma}\big\}
\cdot \frac{\partial^2 u_0^\gamma}{\partial x_k\partial x_j}.
\end{equation}
\end{prop}

\begin{proof}
Note that
\begin{equation}\label{2.4.1}
\aligned
\left(\frac{\partial w_\varep}{\partial \nu_\varep}\right)^\alpha
&=\left(\frac{\partial u_\varep}{\partial \nu_\varep}\right)^\alpha
-\left(\frac{\partial u_0}{\partial \nu_\varep}\right)^\alpha
-n_i a_{ij}^{\alpha\beta} (x/\varep) \frac{\partial}{\partial x_j}
\big[ V_{\varep, k}^{\beta\gamma} -P_k^{\beta\gamma} \big] \cdot \frac{\partial u_0^\gamma}
{\partial x_k}\\
& \qquad\qquad\qquad
-n_i a_{ij}^{\alpha\beta}(x/\varep)
\big[V_{\varep, k}^{\beta\gamma} -P_k^{\beta\gamma} \big]\cdot \frac{\partial^2 u_0^\gamma}
{\partial x_j\partial x_k}.\\
\endaligned
\end{equation}
Since $\frac{\partial}{\partial \nu_\varep}
\big\{ V_{\varep, k}^\beta\big\} =\frac{\partial}{\partial\nu_0} \big\{P_k^\beta\big\}$,
the third term in the right hand side of (\ref{2.4.1})
equals $-\left(\frac{\partial u_0}{\partial \nu_0}\right)^\alpha
+\left(\frac{\partial u_0}{\partial \nu_\varep}\right)^\alpha$.
This gives (\ref{conormal-formula}).
\end{proof}

The following two propositions provide the properties
of the Dirichlet and Neumann correctors needed in this paper.

\begin{prop}\label{prop-2.2}
Let $\mathcal{L}_\varep=-\text{\rm div} \big(A(x/\varep)\nabla\big)$
with $A(y)$ satisfying (\ref{ellipticity}), (\ref{periodicity}) and (\ref{smoothness}).
Let $\big (\Phi_{\varep, j}^\beta\big)$ denote the matrix of Dirichlet correctors
for $\mathcal{L}_\varep$ in a $C^{1,\eta}$ domain $\Omega$.
Then
\begin{equation}\label{2.2.2}
 |\nabla \Phi_j^\beta (x)|\le C, \qquad
 |\Phi_{\varep, j}^\beta (x)   -P_j^\beta (x)|\le C\varep
\end{equation}
and
\begin{equation}\label{2.2.3}
 \big|\nabla \big\{ \Phi_{\varep, j}^\beta (x) - P_j^\beta (x)-\varep \chi_j^\beta \left
({x}/{\varep}\right)\big\}\big| \le C \min \left\{ 1, \frac{\varep}{\delta (x)}\right\}
\end{equation}
for any $x\in \Omega$, where 
$\delta(x)=\text{\rm dist} (x,\partial\Omega)$
and $C$ depends only $d$, $m$, $\mu$, $\lambda$, $\tau$ and $\Omega$.
\end{prop}

\begin{proof}
The first estimate in (\ref{2.2.2}) follows from the Lipschitz estimate in \cite{AL-1987}.
To see the second estimate, let $u_\varep (x) =\Phi_{\varep, j}^\beta (x) -P_j^\beta (x)
-\varep \chi_j^\beta (x/\varep)$. Then $\mathcal{L}_\varep (u_\varep)=0$
in $\Omega$ and $u_\varep (x)=-\varep \chi_j^\beta(x/\varep)$ for $x\in \partial\Omega$.
It again follows from \cite{AL-1987} that $\|u_\varep \|_{L^\infty(\Omega)}
\le C \| u_\varep \|_{L^\infty(\partial\Omega)} \le C \varep$.
Hence, $\|\Phi_{\varep, j}^\beta -P_j^\beta\|_{L^\infty(\Omega)}
\le C\varep +\| \varep \chi_j^\beta\|_\infty \le C\varep$.
Finally, note that $\|\nabla u_\varep\|_{L^\infty(\Omega)} \le C$. 
Also, by the interior estimate in \cite{AL-1987} and $\|u_\varep\|_{L^\infty(\Omega)}
\le C\varep$, one obtains 
$|\nabla u_\varep (x)|\le C \varep [\delta(x)]^{-1}$.
This gives the estimate (\ref{2.2.3}).
\end{proof}

\begin{prop}\label{prop-2.3}
Let $\mathcal{L}_\varep=-\text{\rm div} \big(A(x/\varep)\nabla\big)$
with $A(y)$ satisfying (\ref{ellipticity}), (\ref{periodicity}),
(\ref{smoothness}) and the symmetry condition $A^*=A$.
Let $\big (\Psi_{\varep, j}^\beta\big)$ denote the matrix of Neumann correctors
for $\mathcal{L}_\varep$ in a $C^{1,\eta}$ domain $\Omega$.
Suppose $\Psi_{\varep,j}^\beta (x_0)=P_j^\beta (x_0)$ for some $x_0\in \Omega$.
Then
\begin{equation}\label{2.2.4}
 |\nabla \Psi_j^\beta (x)|\le C, \qquad
 |\Psi_{\varep, j}^\beta (x)   -P_j^\beta (x)|\le C\varep \ln [\varep^{-1}M +2]
\end{equation}
and
\begin{equation}\label{2.2.5}
 \big|\nabla \big\{ \Psi_{\varep, j}^\beta (x) - P_j^\beta (x)-\varep \chi_j^\beta \left
({x}/{\varep}\right)\big\}\big| \le C \min \left\{ 1, \frac{\varep}{\delta (x)}\right\}
\end{equation}
for any $x\in \Omega$, where $M=\text{\rm diam}(\Omega)$ and
 $C$ depends only $d$, $m$, $\mu$, $\lambda$, $\tau$ and $\Omega$.
\end{prop}

\begin{proof} 
The estimate (\ref{2.2.5}) as well as the first estimate in (\ref{2.2.4})
was proved in \cite{KLS1}. To prove the second estimate in (\ref{2.2.4}),
we let $H_{\varep, j}^\beta =\Psi_{\varep, j}^\beta (x)-P_j^\beta (x)
-\varep \chi_j^\beta (x/\varep)$.
Since $|\nabla H_{\varep, j}^\beta (x)|
\le C \min \big(1, \varep [\delta (x)]^{-1}\big)$,
by the Fundamental Theorem of Calculus, we may deduce that
$|H_{\varep, j}^\beta (x)-H_{\varep, j}^\beta (y)|\le C\varep \ln [\varep^{-1} M +2]$
for any $x, y\in \Omega$. Since $|H_{\varep, j}^\beta (x_0)|=\varep |\chi_j^\beta (x_0)|$,
we obtain $|H_{\varep, j}^\beta (x)|\le C\varep \ln [\varep^{-1} M +2]$
for any $x\in \Omega$. This gives the desired estimate.
\end{proof}

\section{Asymptotic behavior of Green functions}

The goal of this section is to prove Theorem \ref{theorem-A}.
We also establish several convergence theorems for
solutions with Dirichlet boundary conditions.

Let $\mathcal{L}_\varep =-\text{\rm div} \big( A(x/\varep) \nabla\big)$
with $A(y)$ satisfying conditions (\ref{ellipticity}), (\ref{periodicity})
and (\ref{smoothness}).
Let $G_\varep(x,y)$ denote the matrix of
Green's functions for $\mathcal{L}_\varep$ in a bounded
domain $\Omega$. It follows from \cite{AL-1987} that
if $\Omega$ is $C^{1,\eta}$ for some $\eta\in (0,1)$,
\begin{equation}\label{Green's-size-estimate-1}
\left\{
\aligned
|G_\varep (x,y)| &\le \frac{C}{|x-y|^{d-2}}\min \left\{ 1, \frac{\delta(x)}{|x-y|},
\frac{\delta(y)}{|x-y|}, \frac{\delta(x)\delta(y)}{|x-y|^2}\right\},\\
|\nabla_x G_\varep (x,y)| & \le \frac{C}{|x-y|^{d-1}}
\min\left\{ 1, \frac{\delta(y)}{|x-y|}\right\},\\
|\nabla_y G_\varep (x,y)| & \le \frac{C}{|x-y|^{d-1}}
\min\left\{ 1, \frac{\delta(x)}{|x-y|}\right\},\\
|\nabla_x\nabla_y G_\varep (x,y)| & \le \frac{C}{|x-y|^d}
\endaligned
\right.
\end{equation}
for any $x,y\in \Omega$,
where $\delta (x)=\text{\rm dist} (x, \partial\Omega)$.
These estimates, which are well known for second-order
elliptic operators with constant coefficients,
play an essential role in our approach to Theorem \ref{theorem-A}.

\subsection{ $L^\infty$ estimates}

In this subsection we give the proof of the estimate (\ref{Green's-size}).
As a corollary of (\ref{Green's-size}), we also establish an $O(\varep)$ estimate
for $\|u_\varep -u_0\|_{L^p(\Omega)}$ for any $p>1$ (see Theorem \ref{theorem-3.1.2}).
Throughout the subsection we will assume that
 $A(y)$ satisfies conditions (\ref{ellipticity})-(\ref{periodicity}) and in the case $m>1$,
$A(y)$ is H\"older continuous. 
Let 
$$
D_r=D_r(x_0,r)=B(x_0,r)\cap\Omega \quad \text{ and } \quad \Delta_r =\Delta (x_0,r)
=B(x_0, r)\cap \partial\Omega
$$
for some $x_0\in \overline{\Omega}$ and $0<r<r_0$.

\begin{lemma}\label{lemma-3.1.0}
Assume that $\Omega$ is Lipschitz if $m=1$, and $C^{1,\eta}$ for some $\eta\in (0,1)$ if $m>1$.
Then
\begin{equation}\label{boundary-L-infty}
\| u_\varep\|_{L^\infty(D_r)}
\le C\| f\|_{L^\infty (\Delta_{3r})}
+ \frac{C}{r^d}
\int_{D_{3r}} |u_\varep|\, dx,
\end{equation}
where $\mathcal{L}_\varep (u_\varep )= 0$ in $D_{3r}$ and $u_\varep =f$
on $\Delta_{3r}$.
\end{lemma}

\begin{proof}
By rescaling we may assume that $r=1$.
The estimate is well known in the case $m=1$ and follows from the maximum principle
and De Giorgi -Nash estimate.
If $m>1$ and $f=0$, estimate (\ref{boundary-L-infty}) follows directly 
from \cite[Lemma 12]{AL-1987}.
To treat the general case, consider $\mathcal{L}_\varep (v_\varep)=0$ 
in $\widetilde{D}$ with the Dirichlet data
$v_\varep =f$ on $\partial\widetilde{D}\cap \partial\Omega$ and
$v_\varep =0$ on $\partial\widetilde{D}\setminus \partial\Omega$, where
$\widetilde{D}$ is a $C^{1,\eta}$ domain such that
$D_{2}\subset \widetilde{D}\subset D_{3}$.
Note that $\|v_\varep\|_{L^\infty(D_{2})} \le C\| f\|_{L^\infty(\Delta_{3})}$
by \cite[Theorem 3]{AL-1987}, and $u_\varep-v_\varep$ may be handled
by \cite[Lemma 12]{AL-1987}, as before.
 \end{proof}

The next lemma provides a boundary $L^\infty$ estimate.

\begin{lemma}\label{lemma-3.1.1}
Assume that $\Omega$ satisfies the same assumption as in Lemma \ref{lemma-3.1.0}.
Let $u_\varep \in H^1(D_{4r})$ and $u_0\in W^{2,p}(D_{4r})$ for some $d<p\le \infty$. Suppose that
$$
\mathcal{L}_\varep (u_\varep) =\mathcal{L}_0 (u_0) \quad \text{ in } D_{4r}
\quad \text{ and } \quad 
u_\varep =u_0 \quad \text{ on } \Delta_{4r}.
$$
Then,
\begin{equation}\label{estimate-3.1.1}
\aligned
\| u_\varep -u_0\|_{L^\infty (D_r)}
\le   \frac{C}{r^d} \int_{D_{4r}} |u_\varep - & u_0|\, dx
+C\varep \|\nabla u_0\|_{L^\infty (D_{4r})}\\
& +C_p\, \varep r^{1-\frac{d}{p}} \, \|\nabla^2 u_0\|_{L^p (D_{4r})}.
\endaligned
\end{equation}
\end{lemma}

\begin{proof}
Note that if $\mathcal{L}_\varep (u_\varep)=F$, then
$\mathcal{L}_{\varep/r} (v)=F_1$, where
$v(x)=r^{-2} u_\varep (rx)$ and $F_1 (x)=F(rx)$.
Thus, by rescaling, it suffices to consider the case $r=1$.
To this end we choose a domain $\widetilde{D}$, which is $C^{1, \eta}$ for $m>1$
and Lipschitz for $m=1$,
such that $D_{3}\subset \widetilde{D}\subset
D_{4}$.
Consider
$$
w_\varep =u_\varep -u_0 -\varep\chi_j^\beta \left({x}/{\varep}\right)
 \frac{\partial u_0^\beta}
{\partial x_j}
 =w_\varep^{(1)} +w_\varep^{(2)} \quad \text{ in } \widetilde{D},
 $$
 where
\begin{equation}\label{definition-w-1}
 \mathcal{L}_\varep (w_\varep^{(1)})  =\mathcal{L}_\varep (w_\varep) \quad \text{ in } \widetilde{D}
\quad \text{ and } \quad w_\varep^{(1)}\in H^1_0 (\widetilde{D})
\end{equation}
and
\begin{equation}\label{definition-w-2}
\mathcal{L}_\varep (w_\varep^{(2)}) =0 \quad \text{ in } \widetilde{D}
\quad \text{ and } \quad
w_\varep^{(2)} =w_\varep \quad \text{ on } \partial \widetilde{D}.
\end{equation}
Since $w_\varep^{(2)}=w_\varep=-\varep \chi (x/\varep)\nabla u_0$ 
on $\Delta_{3}$ and $\|\chi\|_\infty\le C$, 
it  follows from Lemma \ref{lemma-3.1.0} that
$$
\aligned
\|w_\varep^{(2)}\|_{L^\infty(D_1)}
&\le C\varep \|\nabla u_0\|_{L^\infty(\Delta_{3})}
+C\int_{D_{3}} |w_\varep^{(2)}|\, dx \\
&\le  C\varep \|\nabla u_0\|_{L^\infty(\Delta_{3})}
+C\int_{D_{3}} |w_\varep|\, dx
+  C\int_{D_{3}} |w_\varep^{(1)}|\, dx\\
&\le C\int_{D_{3}} |u_\varep -u_0|\, dx
+C\varep \|\nabla u_0\|_{L^\infty(D_{3})} 
+ C\|w_\varep^{(1)}\|_{L^\infty(D_{3})}.
\endaligned
$$
This gives
\begin{equation}\label{3.1.1}
\| u_\varep -u_0\|_{L^\infty(D_1)}\\
 \le C\int_{D_{3}} |u_\varep -u_0|\, dx
+C\varep \|\nabla u_0\|_{L^\infty(D_{3})} 
+ C\|w_\varep^{(1)}\|_{L^\infty(D_{3})}.
\end{equation}

To estimate $w_\varep^{(1)}$ in $D_{3}$, we use the Green function representation
$$
w_\varep^{(1)} (x)
=\int_{\widetilde{D}}
\widetilde{G}_\varep (x,y) \mathcal{L}_\varep (w_\varep) (y)\, dy,
$$
where $\widetilde{G}_\varep (x,y)$ denotes
 the matrix of Green functions for $\mathcal{L}_\varep$
in $\widetilde{D}$.
Using (\ref{formula-2.1}) with $V_{\varep, j}^\beta
=P_j^\beta (x)+\varep \chi_j^\beta (x/\varep)$, we obtain
$$
w_\varep^{(1)} (x)
= -\varep \int_{\widetilde{D}}
\frac{\partial}{\partial y_i} \big\{ \widetilde{G}_\varep (x,y)\big\}
\cdot  \left[ F_{jik} \left({y}/{\varep}\right)
+a_{ij}\left({y}/{\varep}\right) \chi_k \left({y}/{\varep}\right) \right]
\cdot \frac{\partial^2 u_0}{\partial y_j \partial y_k}\, dy,
$$
where we have suppressed the superscripts for notational simplicity.
Note that by Remark \ref{remark-2.1}, $\|F_{jik}\|_\infty \le C$.
It follows that
\begin{equation}\label{3.1.3}
\aligned
|w_\varep^{(1)} (x)|
\le & C\varep \int_{\widetilde{D}} |\nabla_y \widetilde{G}_\varep (x,y)|\, |\nabla^2 u_0 (y)|\, dy\\
&\le C\varep \|\nabla^2 u_0\|_{L^p(D_{4})}
\left\{ \int_{\widetilde{D}} |\nabla_y \widetilde{G}_\varep (x,y)|^{p^\prime}\, dy\right\}^{1/p^\prime}\\
&\le C_p\, \varep \|\nabla^2 u_0\|_{L^p(D_{4})}
\endaligned
\end{equation}
if $p>d$, 
where we have used H\"older's inequality and $\|\nabla_y \widetilde{G}_\varep(x,\cdot)\|_{L^{p^\prime}
(\wD)}\le C_p$.
This, together with (\ref{3.1.1}), completes the proof.
\end{proof}

We are now ready to prove the first estimate in Theorem \ref{theorem-A}.
Note that in the scalar case $m=1$, no smoothness condition on $A(y)$ is needed
in the following theorem.

\begin{thm}\label{theorem-3.1.1}
Suppose that $A(y)$ satisfies conditions (\ref{ellipticity})-(\ref{periodicity}).
If $m>1$, we also assume that $A(y)$ is H\"older continuous.
Let $\Omega$ be a bounded $C^{1,1}$ domain.
Then 
\begin{equation}\label{estimate-3.1.2}
|G_\varep (x,y) -G_0 (x,y)|\le \frac{C\varep}{|x-y|^{d-1}}\quad
\text{ for any } x,y\in \Omega,
\end{equation}
where $C$ depends only on $d$, $m$, $\mu$, $\Omega$ as well as $\lambda$ and $\tau $ (if $m>1$).
\end{thm}

\begin{proof}
Under the assumptions on $A$ and $\Omega$, the estimates
$|G_\varep (x,y)|\le C\, |x-y|^{2-d}$ and 
$|\nabla_x G_0(x,y)|\le C|x-y|^{1-d}$ hold for any $x,y\in \Omega$ and $\varep\ge 0$.
We now fix $x_0,y_0\in \Omega$ and $r=|x_0-y_0|/8$. Let $f\in C_0^\infty (D(y_0,r))$,
$$
u_\varep (x) =\int_\Omega G_\varep (x,y) f(y)\, dy \quad \text{ and }\quad
 u_0 (x) =\int_\Omega G_0 (x,y) f(y)\, dy.
$$
Then $\mathcal{L}_\varep (u_\varep) =\mathcal{L}_0 (u_0)=f$ in $\Omega$ and
$u_\varep =u_0 =0$ on $\partial\Omega$. Also, note that since $\Omega$ is $C^{1,1}$,
\begin{equation}\label{3.1.1-0}
\left\{
\aligned
&\|\nabla^2 u_0\|_{L^p(\Omega)} \le C_p \| f\|_{L^p(D(y_0,r))}\quad
\text{ for any } 
1<p<\infty,\\ 
&\|\nabla u_0\|_{L^\infty(\Omega)}
\le C_p\,  r^{1-\frac{d}{p}} \| f\|_{L^p(D(y_0, r))}\quad\text{ for any } p>d
\endaligned
\right.
\end{equation}
(see e. g. \cite{Gilbarg-Trudinger}).

Next, let
$$
w_\varep =u_\varep -u_0-\varep\chi_j^\beta \left({x}/{\varep}\right)
\frac{\partial u_0^\beta}{\partial x_j}
=\theta_\varep (x) +z_\varep (x),
$$
where $\theta_\varep\in H_0^1(\Omega)$ and $\mathcal{L}_\varep (\theta_\varep)
=\mathcal{L}_\varep (w_\varep)$ in $\Omega$.  Observe that by (\ref{formula-2.1}) 
with $V_{j,\varep}^\beta =P_j^\beta (x)+\varep\chi_j^\beta (x/\varep)$ and
Remark \ref{remark-2.1},
$\|\theta_\varep\|_{H^1_0(\Omega)} \le C\, \varep \| \nabla^2 u_0\|_{L^2(\Omega)}
\le C\, \varep \| f\|_{L^2(D(y_0,r_0))}$.
By H\"older's and Sobolev's inequalities, this implies that
\begin{equation}\label{3.1.1-1}
\|\theta_\varep\|_{L^2(D(x_0,r))}
\le C\, r \|\theta_\varep\|_{L^q(\Omega)}
\le C\, \varep r\| f\|_{L^2(D(y_0, r))}
\le C\, \varep r^{1+\frac{d}{2}-\frac{d}{p}} \| f\|_{L^p(D(y_0,r))},
\end{equation}
where $q=\frac{2d}{d-2}$ and $p>d$.
Also, note that since $\mathcal{L}_\varep (z_\varep)=0$ in $\Omega$ and
$z_\varep =w_\varep$ on $\partial\Omega$, 
\begin{equation}\label{3.1.1-2}
\|z_\varep\|_{L^\infty(\Omega)}\le C\| z_\varep\|_{L^\infty(\partial\Omega)}
\le C\, \varep \|\nabla u_0\|_{L^\infty(\partial\Omega)}.
\end{equation}
In view of (\ref{3.1.1-0})-(\ref{3.1.1-2}), we obtain
\begin{equation}\label{3.1.1-3}
\aligned
\| u_\varep -u_0\|_{L^2(D(x_0,r))}
&\le \|\theta_\varep\|_{L^2(D(x_0,r))} +\|z_\varep\|_{L^2(D(x_0,r))} +
C\, \varep r^{\frac{d}{2}}\|\nabla u_0\|_{L^\infty(\Omega)}\\
&\le \|\theta_\varep\|_{L^2(D(x_0,r))} +
C\, \varep r^{\frac{d}{2}}\|\nabla u_0\|_{L^\infty(\Omega)}\\
&\le C\, \varep r^{1+\frac{d}{2}-\frac{d}{p}} \| f\|_{L^p(D(y_0, r))},
\endaligned
\end{equation}
where $p>d$.
This, together with Lemma \ref{lemma-3.1.1} and (\ref{3.1.1-0}), gives
$$
|u_\varep (x_0)-u_0(x_0)| \le C_p\, \varep r^{1-\frac{d}{p}} \| f\|_{L^p(D(y_0, r))}.
$$
It then follows by duality that
$$
\left\{ \int_{D(y_0,r)} |G_\varep(x_0, y)-G_0(x_0, y)|^{p^\prime}\, dy\right\}^{1/p^\prime}
\le C_p\, \varep r^{1-\frac{d}{p}} \quad \text{ for any } p>d.
$$

Finally, since $\mathcal{L}_\varep^* \big( G_\varep (x_0, \cdot)\big)
=\mathcal{L}_0^* \big(G_0(x_0, \cdot)\big)=0$ in $D(y_0,r)$,
we may invoke Lemma \ref{lemma-3.1.1} again to conclude that
$$
\aligned
 |G_\varep (x_0, y_0)-G_0 (x_0, y_0)|
& \le \frac{C}{r^d}
\int_{D(y_0, r)} |G_\varep ( x_0, y)-G_0 (x_0, y)|\, dy\\
&\qquad\qquad  + C\,\varep \|\nabla_y G_0(x_0,\cdot)\|_{L^\infty (D(y_0, r))}\\
&\qquad \qquad +C_p \, \varep r^{1-\frac{d}{p}} \|\nabla^2_y G_0(x_0, \cdot)\|_{L^p(D(y_0, r_0))}\\
&\le C\, \varep r^{1-d},
\endaligned
$$
where we also used 
$$
\left\{ \frac{1}{r^d}
\int_{D(y_0,r)} |\nabla_y^2 G_0(x_0,y)|^p\, dy\right\}^{1/p}
\le C_p\, r^{-2} \| G_0(x_0, \cdot)\|_{L^\infty (D(y_0, 2r))}
\le C_p\, r^{-d},
$$
obtained by the boundary $W^{2,p}$ estimates on $C^{1,1}$ domains \cite{Gilbarg-Trudinger}.
This completes the proof.
\end{proof}

As a corollary of estimate (\ref{estimate-3.1.2}), we obtain an $O(\varep)$ estimate
for $\|u_\varep -u_0\|_{L^p(\Omega)}$ for any $p>1$. 
In particular, we recover the estimate $\| u_\varep-u_0\|_{L^2(\Omega)} 
\le C \varep \| F\|_{L^2(\Omega)}$, 
proved in \cite{Griso-2006} for scalar equations with bounded measurable coefficients satisfying
(\ref{ellipticity})-(\ref{periodicity}).
Also see \cite{KLS2} for estimates of $\|u_\varep-u_0\|_{L^2(\Omega)}$
on Lipschitz domains.
 
\begin{thm}\label{theorem-3.1.2}
Suppose that $A(y)$ and $\Omega$ satsify the same conditions as in Theorem \ref{theorem-3.1.1}.
For $F\in L^2(\Omega)$ and $\varep\ge 0$, let $u_\varep \in H_0^1(\Omega)$ be the solution
of $\mathcal{L}_\varep (u_\varep)=F$ in $\Omega$.
Then the estimate
\begin{equation}\label{estimate-3.1.7}
\| u_\varep -u_0\|_{L^q (\Omega)} \le C\varep \| F\|_{L^p(\Omega)}
\end{equation}
holds if $1<p<d$ and $\frac{1}{q}=\frac{1}{p}-\frac{1}{d}$, or $p>d$ and $q=\infty$.
Moreover,
\begin{equation}\label{estimate-3.1.8}
\| u_\varep -u_0\|_{L^\infty(\Omega)} 
\le C\varep \big[\ln \big(\varep^{-1}M +2\big)\big]^{1-\frac{1}{d}} \| F\|_{L^d 
(\Omega)},
\end{equation}
where $M=\text{\rm diam} (\Omega)$.
\end{thm}

\begin{proof}
It follows from the Green function representation and Theorem \ref{theorem-3.1.1} that
$$
|u_\varep (x)-u_0(x)|\le C\, \varep \int_\Omega \frac{|F(y)|}{|x-y|^{d-1}}\, dy,
\quad \text{ for any } x\in \Omega.
$$
This leads to (\ref{estimate-3.1.7}) for $1<p<d$ and $\frac{1}{q}=\frac{1}{p}
-\frac{1}{d}$ by the well known estimates for fractional integrals.
The case of $p>d$ and $q=\infty$ follows directly from H\"older's inequality.
To see (\ref{estimate-3.1.8}), we bound $|G_\varep (x,y)-G_0(x,y)|$ by $ C|x-y|^{2-d}$ if
$|x-y|<\varep$, and by $C\varep |x-y|^{1-d}$ if $|x-y|\ge \varep$.
By H\"older's inequality, this gives
$$
\aligned
|u_\varep (x)-u_0(x)| &\le C\int_{D(x,\varep)} \frac{|F(y)|}{|x-y|^{d-2}}\, dy
+ C\varep \int_{\Omega\setminus D(x,\varep)} \frac{|F(y)|}{|x-y|^{d-1}}\, dy\\
& \le C\varep \| F\|_{L^d(\Omega)}
+C\varep \big[\ln \left( \varep^{-1}M +2\right)\big]^{1-\frac{1}{d}} \| F\|_{L^d(\Omega)}\\
&\le C\varep \big[\ln \left( \varep^{-1}M +2\right)\big]^{1-\frac{1}{d}}
 \| F\|_{L^d(\Omega)},
\endaligned
$$
which completes the proof.
\end{proof}

\subsection{Lipschitz estimates}

In this subsection we give the proof of (\ref{Green's-derivative}).
As a corollary of (\ref{Green's-derivative}), we also obtain an $O(\varep)$
estimate for $u_\varep -u_0-\{ \Phi_{\varep,j} -P_j\} \frac{\partial u_0}{\partial x_j}$
in ${W^{1,p}_0(\Omega)}$
for any $1<p<\infty$.

Recall that $D(r)=D(x_0,r)=B(x_0,r)\cap\Omega$ and $\Delta_r =\Delta(x_0,r)
=B(x_0, r)\cap\partial\Omega$, where $x_0\in \overline{\Omega}$ and $0<r<r_0$.
Throughout this subsection we will assume that
$\Omega$ is a bounded $C^{2,\eta}$ domain for some $\eta\in (0,1)$ and 
$A=A(y)$ satisfies conditions (\ref{ellipticity}), (\ref{periodicity}) and
(\ref{smoothness}).

\begin{lemma}\label{lemma-3.2.1}
Suppose that $u_\varep\in H^1(D_{4r})$, 
$u_0 \in C^{2,\rho}(D_{4r})$ and $\mathcal{L}_\varep (u_\varep)=\mathcal{L}_0 (u_0)$
in $D_{4r}$, where $0<\rho<\eta$.
Also assume that $u_\varep=u_0$ on $\Delta_{4r}$.
Then, if $0<\varep<r$,
\begin{equation}\label{estimate-3.2.1}
\aligned
 \|\frac{\partial u^\alpha_\varep}{\partial x_i}
& -\frac{\partial}{\partial x_i} \big\{ \Phi_{\varep, j}^{\alpha\beta}\big\}
\cdot \frac{\partial u_0^\beta}{\partial x_j}\|_{L^\infty (D_r)}\\
\le & \frac{C}{r^{d+1}}\int_{D_{4r}} |u_\varep -u_0|\, dx
 +C\varep r^{-1} \| \nabla u_0\|_{L^\infty(D_{4r})}\\
& +C\varep \ln \big[ \varep^{-1} r +2\big] \| \nabla^2 u_0\|_{L^\infty (D_{4r})}
+C \varep r^\rho \| \nabla^2 u_0\|_{C^{0,\rho} (D_{4r})},
\endaligned
\end{equation}
where $\Phi_\varep= \big( \Phi_{\varep, j}^{\alpha\beta}\big)$ denotes the matrix of Dirichlet correctors
for $\mathcal{L}_\varep$ in $\Omega$.
\end{lemma}

\begin{proof}
We begin by choosing a $C^{2,\eta}$ domain
$\widetilde{D}$ such that $D_{3r}\subset \widetilde{D}\subset D_{4r}$.
Let
$$
w_\varep =u_\varep -u_0 -\big\{ \Phi_{\varep, j}^\beta -P_j^\beta\big\}\cdot \frac{\partial u_0^\beta}
{\partial x_j}.
$$
Note that $w_\varep=0$ on $\Delta_{4r}$.
Write $w_\varep =w_\varep^{(1)} +w_\varep^{(2)}$ in $\widetilde{D}$, where $w_\varep^{(1)}
\in H_0^1(\widetilde{D})$ and $\mathcal{L}_\varep (w_\varep^{(1)}) =\mathcal{L}_\varep (w_\varep)$
in $\widetilde{D}$.
Since $\mathcal{L}_\varep (w_\varep^{(2)})=0$ in $D_{3r}$
and $w_\varep^{(2)}=w_\varep =0$ on $\Delta_{3r}$, it follows from the boundary
Lipschitz estimate in \cite[Lemma 20]{AL-1987} that
$$
\aligned
\|\nabla w_\varep^{(2)}\|_{L^\infty(D_r)}
&\le \frac{C}{r^{d+1}}
\int_{D_{2r}} |w_\varep^{(2)}|\, dx\\
&\le \frac{C}{r^{d+1}}
\int_{D_{2r}} |w_\varep|\, dx
+Cr^{-1} \|w_\varep^{(1)}\|_{L^\infty(D_{2r})}\\
&\le \frac{C}{r^{d+1}}
\int_{D_{2r}} |u_\varep -u_0|\, dx
+C\varep r^{-1} \|\nabla u_0\|_{L^\infty(D_{2r})}
+Cr^{-1} \|w_\varep^{(1)}\|_{L^\infty(D_{2r})},
\endaligned
$$
where we have used the estimate $\|\Phi_{\varep, j}^\beta -P_j^\beta\|_\infty\le C\varep$
in Proposition \ref{prop-2.2}.
This implies that
$$
\|\nabla w_\varep\|_{L^\infty(D_r)}
\le 
\frac{C}{r^{d+1}}
\int_{D_{2r}} |u_\varep -u_0|\, dx
+C\varep r^{-1} \|\nabla u_0\|_{L^\infty(D_{2r})}
+C \|\nabla w_\varep^{(1)}\|_{L^\infty(D_{2r})},
$$
where we have used $\|w_\varep^{(1)}\|_{L^\infty(D_{2r})}
\le Cr \| \nabla w_\varep^{(1)}\|_{L^\infty(D_{2r})}$.
Thus,
\begin{equation}\label{3.2.10}
\aligned
\|\frac{\partial u^\alpha_\varep}{\partial x_i}
& -\frac{\partial}{\partial x_i} \big\{ \Phi_{\varep, j}^{\alpha\beta}\big\}
\cdot \frac{\partial u_0^\beta}{\partial x_j}\|_{L^\infty (D_r)}\\
& \le 
\frac{C}{r^{d+1}}
\int_{D_{2r}} |u_\varep -u_0|\, dx
+C\varep r^{-1} \|\nabla u_0\|_{L^\infty(D_{2r})}\\
&\quad\quad\quad +C\varep \| \nabla^2 u_0\|_{L^\infty(D_{2r})}
+C \|\nabla w_\varep^{(1)}\|_{L^\infty(D_{2r})}.
\endaligned
\end{equation}

To estimate $\nabla w_\varep^{(1)}$ on $D_{2r}$, we use the Green function representation
$$
w_\varep^{(1)} (x)
=\int_{\widetilde{D}} \widetilde{G}_\varep (x,y) \mathcal{L}_\varep (w_\varep) (y)\, dy,
$$
where
$\widetilde{G}_\varep (x,y)$ is the matrix of Green functions for $\mathcal{L}_\varep$ in 
the $C^{2,\eta}$ domain $\widetilde{D}$.
Let 
$$f_i (x) =\varep F_{kij} \left({x}/{\varep}\right) \frac{\partial^2 u_0}
{\partial x_j\partial x_k}
+a_{ij}\left({x}/{\varep}\right) \big[
\Phi_{\varep, k}-P_k\big] \cdot \frac{\partial^2 u_0}{\partial x_j\partial x_k},
$$
where we have suppressed the superscripts for notational simplicity. In view of (\ref{formula-2.1}),
we obtain
$$
\aligned
w_\varep^{(1)} (x)
=& -\int_{\widetilde{D}} \frac{\partial}{\partial y_i}
\big\{ \widetilde{G}_\varep (x,y)\big\} \cdot \big\{ f_i (y)-f_i (x)\big\}\, dy\\
&+\int_{\widetilde{D}}
\widetilde{G}_\varep (x,y) a_{ij}\left({y}/{\varep}\right)
\frac{\partial}{\partial y_j} \big[ \Phi_{\varep, k} -P_k -\varep \chi_k \left({y}/{\varep}\right)
\big] \cdot \frac{\partial^2 u_0}{\partial y_i \partial y_k}\, dy.
\endaligned
$$
It follows that
\begin{equation}\label{3.2.20}
\aligned
|\nabla w_\varep^{(1)}(x)|
&\le \int_{\widetilde{D}} |\nabla_x\nabla_y \widetilde{G}_\varep (x,y)|\, |f(y)-f(x)|\, dy\\
&
+C\| \nabla^2 u_0\|_{L^\infty (D_{4r})}
\int_{\widetilde{D}} |\nabla_x \widetilde{G}_\varep (x,y)|\, \big|\nabla_y \big[
\Phi_{\varep, j} -P_j -\varep \chi_j \left({y}/{\varep}\right)\big]\big|\, dy.
\endaligned
\end{equation}
To handle the first term in the right hand side of (\ref{3.2.20}), we use
$|\nabla_x\nabla_y \widetilde{G}_\varep (x,y)|\le C|x-y|^{-d}$ and the observation that
$$
\aligned
\| f\|_{L^\infty(D_{4r})} & \le C \, \varep \|\nabla^2 u_0\|_{L^\infty (D_{4r})},\\
 | f(x)-f(y)| & \le C |x-y|^\rho
\left\{ \varep^{1-\rho} \|\nabla^2 u_0\|_{L^\infty (D_{4r})}
+\varep \|\nabla^2 u_0\|_{C^{0,\rho} (D_{4r})}\right\}.
\endaligned
$$
This yields that
$$
\aligned
& \int_{\widetilde{D}} |\nabla_x\nabla_y \widetilde{G}_\varep (x,y)| |f(y)-f(x)|\, dy\\
& \qquad \le C\, \varep \|\nabla^2 u_0\|_{L^\infty (D_{4r})}
\int_{\widetilde{D}\setminus B(x,\varep)} \frac{dy}{|x-y|^d}\\
& \qquad \qquad + C \left\{ \varep^{1-\rho} \|\nabla^2 u_0\|_{L^\infty (D_{4r})}
+\varep \|\nabla^2 u_0\|_{C^{0,\rho} (D_{4r})} \right\}
\int_{\widetilde{D}\cap B(x,\varep)} 
\frac{dy}{|x-y|^{d-\rho}}\\
&\qquad \le C\, \varep \ln [\varep^{-1} r +2]\|\nabla^2 u_0\|_{L^\infty (D_{4r})}
+C\varep^{1+\rho} \|\nabla^2 u_0\|_{C^{0,\rho} (D_{4r})}.
\endaligned
$$

Finally, using the estimates
$|\nabla_x \widetilde{G}_\varep (x,y)|\le C \text{dist} (y, \partial\Omega) |x-y|^{-d}$
and $|\nabla_x \widetilde{G}_\varep (x,y)|
\le C |x-y|^{1-d}$ as well as estimates in Proposition \ref{prop-2.2}, we see that
 the second term in the right-hand side of (\ref{3.2.20}) is bounded by
$$
\aligned
& C \|\nabla^2 u_0\|_{L^\infty (D_{4r})} 
\left\{ \varep \int_{\widetilde{D}\setminus B(x,\varep)} 
\frac{dy}{|x-y|^d}
+\int_{\widetilde{D}\cap B(x,\varep)} \frac{dy}{|x-y|^{d-1}}\right\}\\
& \le C\, \varep \ln [\varep^{-1} r +2] \|\nabla^2 u_0\|_{L^\infty (D_{4r})}.
\endaligned
$$
Thus we have proved that
$$
\|\nabla w_\varep^{(1)} \|_{L^\infty(D_{3r})}
\le C\, \varep \ln [\varep^{-1} r +2]\|\nabla^2 u_0\|_{L^\infty (D_{4r})}
+C\varep^{1+\rho} \|\nabla^2 u_0\|_{C^{0,\rho} (D_{4r})}.
$$
This, together with (\ref{3.2.10}), completes the proof of (\ref{estimate-3.2.1}).
\end{proof}

We are now ready to give the proof of the estimate (\ref{Green's-derivative}).

\begin{thm}\label{theorem-3.2.1}
Suppose that $A(y)$ satisfies conditions (\ref{ellipticity}), (\ref{periodicity}) and
(\ref{smoothness}).
Let $\Omega$ be a bounded $C^{2,\eta}$ domain for some $\eta\in (0,1)$.
Then
\begin{equation}\label{estimate-3.2.2}
\big| \frac{\partial}{\partial x_i} \big\{ G_\varep^{\alpha\gamma}(x,y)\big\}
-\frac{\partial}{\partial x_i} \big\{ \Phi_{\varep, j}^{\alpha\beta} (x) \big\} \cdot
\frac{\partial }{\partial x_j}\big\{ G_0^{\beta\gamma} (x,y)\big\} \big|
\le \frac{C \, \varep \ln [ \varep^{-1}|x-y| +2]}{|x-y|^d},
\end{equation}
for any $x,y\in \Omega$, where $C$ depends only on $d$, $m$, $\mu$, 
$\lambda$, $\tau$ and $\Omega$.
\end{thm}

\begin{proof}
Fix $x_0$, $y_0\in \Omega$ and $r=|x_0-y_0|/8$.
We may assume that $\varep \le r$, since the case $\varep>r$ is trivial and
follows from the size
estimate of $|\nabla_x G_\varep (x,y)|$, $|\nabla_xG_0(x,y)|$ and (\ref{2.2.2}).

Let $u_\varep (x)=G_\varep (x,y_0)$ and $u_0 (x)=G_0(x,y_0)$.
Then $\mathcal{L}_\varep (u_\varep)=\mathcal{L}_0 (u_0)=0$ in $D_{4r}=D(x_0,4r)$
and $u_\varep=u_0=0$ on $\Delta_{4r} =\Delta(x_0, 4r)$.
Note that by Theorem \ref{theorem-3.1.1}, we have $\| u_\varep -u_0\|_{L^\infty (D_{4r})}
\le C\, \varep r^{1-d}$. 
Also, since $\Omega$ is $C^{2,\eta}$, we have
$\|\nabla u_0\|_{L^\infty (D_{4r})} \le C r^{1-d}$,
$\|\nabla^2 u_0\|_{L^\infty(D_{4r})} \le C r^{-d}$
and $\|\nabla^2 u_0\|_{C^{0, \rho}(D_{4r})} \le Cr^{-d-\rho}$.
It then follows from Lemma \ref{lemma-3.2.1} that
$$
\|\frac{\partial u^\alpha_\varep}{\partial x_i}
-\frac{\partial}{\partial x_i} \big\{ \Phi_{\varep, j}^{\alpha\beta}\big\}
\cdot \frac{\partial u_0^\beta}{\partial x_j} \|_{L^\infty(D_r)}
\le C \varep r^{-d} \ln [ \varep^{-1} r +2].
$$
This finishes the proof.
\end{proof}


As a corollary of the estimate (\ref{estimate-3.2.2}), we obtain an
$O(\varep)$ estimate (up to a logarithmic factor if $p\neq 2$) for 
$u_\varep-u_0 -\{ \Phi_{\varep, j}- P_j\} \frac{\partial u_0}{\partial x_j}$
in $W_0^{1,p}(\Omega)$.

\begin{thm} \label{theorem-3.2.2}
Assume that $A(y)$ and $\Omega$ satisfy the same assumptions as in Theorem \ref{theorem-3.2.1}.
Let $1<p<\infty$.
For $F\in L^p(\Omega)$ and $\varep\ge 0$,
let $u_\varep \in W^{1,p}_0(\Omega)$ and $\mathcal{L}_\varep (u_\varep)=F$ in $\Omega$.
Then
\begin{equation}\label{estimate-3.2.10}
\| u_\varep -u_0 - \big\{ \Phi_{\varep, j}^\beta -P_j^\beta\big\}\frac{\partial u_0^\beta}
{\partial x_j} \|_{W_0^{1,p}(\Omega)}\\
\le C_p\, \varep \big\{ \ln [\varep^{-1}M +2]\big\}^{4|\frac12-\frac{1}{p}|}
 \| F\|_{L^p(\Omega)},
\end{equation}
where $M=\text{\rm diam}(\Omega)$ and $C_p$ depends only
on $d$, $m$, $p$, $\mu$, $\lambda$, $\tau$ and $\Omega$.
\end{thm}

\begin{proof}
We will show that for any $1\le p\le \infty$,
\begin{equation}\label{estimate-3.2.2-0}
\|\frac{\partial u^\alpha_\varep}{\partial x_i}
-\frac{\partial }{\partial x_i} \big\{ \Phi_{\varep, j}^{\alpha\beta}\big\}
\cdot \frac{\partial u_0^\beta}{\partial x_j} \|_{L^p(\Omega)}
\le C \varep \left\{ \ln [ \varep^{-1} M +2]\right\}^{4|\frac12-\frac{1}{p}|}
 \| F\|_{L^p(\Omega)}.
\end{equation}
This, together with (\ref{2.2.2}) and the estimate 
$\|\nabla^2 u_0\|_{L^p(\Omega)}\le C \| F\|_{L^p(\Omega)}$ for $1<p<\infty$,
gives (\ref{estimate-3.2.10}). To see (\ref{estimate-3.2.2-0}), we use
Theorem \ref{theorem-3.2.1} as well as
estimates on $|\nabla_x G_\varep(x,y)|$ and $|\nabla \Phi_\varep|$ to deduce that
$$
\big|\frac{\partial u^\alpha_\varep}{\partial x_i}
-\frac{\partial }{\partial x_i} \big\{ \Phi_{\varep, j}^{\alpha\beta}\big\}
\cdot \frac{\partial u_0^\beta}{\partial x_j}\big|
\le C \int_\Omega K_\varep (x,y) |f(y)|\, dy,
$$
where
\begin{equation}\label{definition-of-K}
K_\varep (x,y)= \left\{
\aligned
& \varep |x-y|^{-d} \ln \big[ \varep^{-1} |x-y|+2], \quad \text{ if } |x-y|\ge \varep,\\
&  |x-y|^{1-d}, \qquad\qquad\qquad\quad  \qquad \text{ if } |x-y|<\varep.
\endaligned
\right.
\end{equation}
Note that
$$
\sup_{x\in \Omega} \int_\Omega K_\varep (x,y)\, dy 
+\sup_{y\in \Omega} \int_\Omega K_\varep (x,y)\, dx 
\le C\varep \big\{ \ln [\varep^{-1}M +2]\big\}^2.
$$
This gives (\ref{estimate-3.2.2-0}) in the case $p=1$ or $\infty$.
Thus, by the M. Riesz interpolation theorem,
 it suffices to prove the estimate for $p=2$.

Let $w_\varep = u_\varep -u_0 -\{ \Phi^\beta_{\varep,j} -P^\beta_j\} 
\frac{\partial u^\beta_0}{\partial x_j}$ and $\delta (x)=\text{dist}(x, \partial\Omega)$.
We may deduce from Propositions \ref{prop-2.1} and \ref{prop-2.2} that
\begin{equation}\label{estimate-3.2.2-1}
\aligned
\int_\Omega |\nabla w_\varep|^2\, dx
& \le C\varep \int_\Omega |\nabla^2 u_0| \, |\nabla w_\varep|\, dx
+C \varep \int_\Omega [\delta(x)]^{-1} |\nabla^2 u_0|\, |w_\varep|\, dx\\
& \le C \varep
\|\nabla^2 u_0\|_{L^2(\Omega)}
\|\nabla w_\varep\|_{L^2(\Omega)},
\endaligned
\end{equation}
where, for the last inequality,
  we have used H\"older's inequality as well as Hardy's inequality
$ \|[\delta(x)]^{-1} w _\varep\|_{L^2(\Omega)}
\le C \|\nabla w_\varep\|_{L^2(\Omega)}$.
The desired estimate follows easily from (\ref{estimate-3.2.2-1}).
\end{proof}

Let $G_\varep^* (x,y)=\big( G_\varep^{*\alpha\beta} (x,y)\big)$ 
denote the matrix of Green's functions for $\mathcal{L}_\varep^*$,
the adjoint of $\mathcal{L}_\varep$. By Theorem \ref{theorem-3.2.1},
\begin{equation}\label{adjoint-estimate-1}
\big| \frac{\partial}{\partial x_i} \big\{ G_\varep^{*\alpha\gamma}(x,y)\big\}
-\frac{\partial}{\partial x_i} \big\{ \Phi_{\varep, j}^{*\alpha\beta} (x) \big\} \cdot
\frac{\partial }{\partial x_j}\big\{ G_0^{*\beta\gamma} (x,y)\big\} \big|
\le \frac{C \, \varep \ln [ \varep^{-1}|x-y| +2]}{|x-y|^d},
\end{equation}
where $\Phi_\varep^*$ denotes the matrix of Dirichlet correctors for $\mathcal{L}_\varep^*$
in $\Omega$.
Since $G_\varep^{*\alpha\beta} (x,y) =G_\varep^{\beta\alpha} (y,x)$, we obtain
\begin{equation}\label{adjoint-estimate-2}
\big| \frac{\partial}{\partial y_i} \big\{ G_\varep^{\gamma\alpha}(x,y)\big\}
-\frac{\partial}{\partial y_i} \big\{ \Phi_{\varep, j}^{*\alpha\beta} (y) \big\} \cdot
\frac{\partial }{\partial y_j}\big\{ G_0^{\gamma\beta} (x,y)\big\} \big|
\le \frac{C \, \varep \ln [ \varep^{-1}|x-y| +2]}{|x-y|^d}.
\end{equation}
This leads to an asymptotic expansion of the Poisson kernel for $\mathcal{L}_\varep$
on $\Omega$.

Let $(h^{\alpha\beta} (y))$ denote the inverse matrix of 
$\big( n_i(y)n_j(y)\hat{a}_{ij}^{\alpha\beta}\big)_{m\times m}$.

\begin{thm}\label{Poisson-kernel-theorem}
Suppose that $A(y)$ satisfies conditions (\ref{ellipticity}), (\ref{periodicity})
and (\ref{smoothness}).
Let $P_\varep (x,y) =\big( P_\varep^{\alpha\beta}(x,y) \big)$ be the Poisson kernel
for $\mathcal{L}_\varep$ on a $C^{2,\eta}$ domain $\Omega$.
Then 
\begin{equation}\label{Poisson-kernel-estimate}
P_\varep^{\alpha\beta} (x,y)=P_0^{\alpha\gamma}(x,y) \omega_\varep^{\gamma\beta}(y)
+R_\varep^{\alpha\beta} (x,y),
\end{equation}
where
\begin{equation}\label{definition-of-omega}
\omega_\varep^{\gamma\beta} (y)
=h^{\gamma \sigma} (y) \cdot \frac{\partial}{\partial n(y)}
\big\{ \Phi_{\varep, k}^{*\rho\sigma} (y) \big\} \cdot n_k(y) \cdot
n_i(y)n_j(y) a_{ij}^{\rho\beta} (y/\varep),
\end{equation}
\begin{equation}\label{remainder-estimate}
|R_\varep^{\alpha\beta} (x,y)|
\le \frac{C \, \varep \ln [ \varep^{-1}|x-y| +2]}{|x-y|^d}
\qquad \text{ for any }
x\in \Omega \text{ and } y\in \partial\Omega, 
\end{equation}
and $C$ depends
only on $d$, $m$, $\mu$, $\lambda$, $\tau$ and $\Omega$.
\end{thm}

\begin{proof}
Note that
\begin{equation}\label{3.2.30}
\aligned
P_\varep^{\alpha\beta} (x,y)
& =-n_i(y) a_{ji}^{\gamma\beta} (y/\varep) \frac{\partial}{\partial y_j}
\big\{ G_\varep^{\alpha\gamma} (x,y)\big\}\\
& =-\frac{\partial}{\partial n(y)} \big\{ G_\varep^{\alpha\gamma} (x,y)\big\}
\cdot n_i(y)n_j(y) a_{ij}^{\gamma\beta} (y/\varep),
\endaligned
\end{equation}
since $G_\varep (x, \cdot)=0$ on $\partial\Omega$. By (\ref{adjoint-estimate-2}),
we obtain
\begin{equation}\label{3.2.31}
\aligned
& \big| P_\varep^{\alpha\beta} (x,y)
+\frac{\partial}{\partial n(y)}
\big\{ G_0^{\alpha \sigma} (x,y)\big\}\cdot
\frac{\partial}{\partial n(y)} \big\{ \Phi_{\varep, k}^{*\gamma \sigma} (y)\big\}
\cdot 
n_i(y)n_j(y) a_{ij}^{\gamma\beta} (y/\varep) n_k (y)\big|\\
&\le \frac{C \, \varep \ln [ \varep^{-1}|x-y| +2]}{|x-y|^d}.
\endaligned
\end{equation}
In view of (\ref{3.2.30}) (with $\varep=0$), we have
$$
P_0^{\alpha\beta} (x,y) h^{\beta\sigma} (y) =-\frac{\partial}{\partial n(y)} 
\big\{ G_0^{\alpha \sigma} (x,y)\big\}.
$$
This, together with (\ref{3.2.31}), gives
$$
|P_\varep^{\alpha\beta} (x,y)- P_0^{\alpha\gamma} (x,y) \omega_\varep^{\gamma\beta} (y)|
\le 
\frac{C \, \varep \ln [ \varep^{-1}|x-y| +2]}{|x-y|^d},
$$
for any $x\in \Omega$ and $y\in \partial\Omega$,
where $\omega_\varep (y)$ is defined by (\ref{definition-of-omega}).
\end{proof}

With the asymptotic expansion for the Poisson kernels at our disposal, we may approximate
the solution of the $L^p$ Dirichlet problem: $\mathcal{L}_\varep (u_\varep)=0$
in $\Omega$ and $u_\varep =f_\varep$ on $\partial\Omega$ by the solution of the homogenized
system with boundary data $\omega_\varep f_\varep$.
As mentioned in the Introduction, the case $f_\varep(x)=f(x, x/\varep)$ 
with $f(x,y)$ periodic in $y$ is of particular interest.

\begin{thm}\label{approximation-thm}
Assume that $A(y)$ and $\Omega$ satisfy the same assumptions as in Theorem \ref{Poisson-kernel-theorem}.
Let $\mathcal{L}_\varep (u_\varep)=0$ in $\Omega$ and $u_\varep =f_\varep$ on $\partial\Omega$.
Then for any $1\le p<\infty$,
\begin{equation}\label{approximation-estimate}
\| u_\varep -v_\varep\|_{L^p(\Omega)}
\le C \left\{ \varep \big( \ln [\varep^{-1}M +2]\big)^2 
\right\}^{1/p} \| f_\varep\|_{L^p(\partial\Omega)},
\end{equation}
where $\mathcal{L}_0 (v_\varep) =0$ in $\Omega$ and $v_\varep =\omega_\varep f_\varep$
on $\partial\Omega$, with $\omega_\varep$ defined by (\ref{definition-of-omega}).
\end{thm}

\begin{proof} Since 
$$
 u_\varep (x) =\int_{\partial\Omega} P_\varep (x,y) f_\varep (y)\, dy\quad
\text{ and } \quad
 v_\varep (x) =\int_{\partial\Omega} P_0 (x,y) \omega_\varep (y) f_\varep (y)\, dy,
$$
it follows from Theorem \ref{Poisson-kernel-theorem} that
$$
|u_\varep(x)-v_\varep(x)|
\le \int_{\partial\Omega} |R_\varep (x,y)|| f_\varep (y)|\, dy.
$$
Using 
$$
|R_\varep (x,y)|\le C\big\{ |\nabla_y G_\varep (x,y)|+|\nabla_y G_0(x,y)|\big\}
\le C\delta (x) |x-y|^{-d},
$$
where $\delta(x)=\text{dist}(x,\partial\Omega)$, we see that
$\int_{\partial\Omega} |R_\varep (x,y)|\, dy\le C$ for any $x\in\Omega$. It then
follows by H\"older's inequality that
\begin{equation}\label{3.2.40}
|u_\varep (x)-v_\varep (x)|^p \le 
C\int_{\partial\Omega} |R_\varep (x,y)|| f_\varep (y)|^p\, dy.
\end{equation}
Now, by Theorem \ref{Poisson-kernel-theorem} as well as the estimate $|R_\varep (x,y)|
\le C|x-y|^{1-d}$, we obtain
$$
\aligned
\int_\Omega |R_\varep (x,y)|\, dx
&\le  C \int_{\Omega\cap B(y,\varep)} \frac{dx}{|x-y|^{d-1}}
+ C \varep \int_{\Omega\setminus B(y,\varep)}
\frac{\ln [\varep^{-1} |x-y|+2]}{|x-y|^d}\, dx\\
& \le C\varep \big\{ \ln [\varep^{-1} M +2]\big\}^2,
\endaligned
$$
for any $y\in \partial\Omega$.
This, together with (\ref{3.2.40}) and Fubini's Theorem, gives (\ref{approximation-estimate}).
\end{proof}

We end this section with another approximation result.
Note that by (\ref{estimate-3.2.2-0}),
\begin{equation}\label{another-approx}
\|\nabla (\mathcal{L}_\varep)^{-1} (F)
-\nabla \Phi_\varep \cdot \nabla (\mathcal{L}_0)^{-1}(F)\|_{L^p(\Omega)}
\le C \varep 
\big\{ \ln [\varep^{-1} M+2]\big\}^{4|\frac12
-\frac{1}{p}|} \| F\|_{L^p(\Omega)},
\end{equation}
for $1\le p\le \infty$.
By duality this gives the following.

\begin{thm}\label{approximation-theorem-2}
Assume that $A(y)$ and $\Omega$ satisfy the same assumptions 
as in Theorem \ref{Poisson-kernel-theorem}.
For $f=(f_i^\alpha)\in L^2(\Omega)$ and $\varep\ge 0$, let
$u_\varep\in H_0^1(\Omega)$ and $\mathcal{L}_\varep (u_\varep) =\text{\rm div} (f)$
in $\Omega$. Then if $f\in L^p(\Omega)$ for some $1\le p\le \infty$,
\begin{equation}\label{estimate-3.2.41}
\| u_\varep -v_\varep\|_{L^p(\Omega)}
\le C\, \varep \big\{ \ln [\varep^{-1} M+2]\big\}^{4|\frac12
-\frac{1}{p}|} \| f\|_{L^p(\Omega)},
\end{equation}
where
$v_\varep\in H_0^1(\Omega)$ and $\mathcal{L}_0 (v_\varep) =\text{\rm div} (F_\varep)$
with 
$$
F_{\varep, i}^\alpha  (x)=f_j^\beta (x)
\frac{\partial}{\partial x_j} \left\{ \Phi_{\varep, i}^{*\beta\alpha}\right\}.
$$
\end{thm}

\subsection{An asymptotic expansion for $\nabla_x\nabla_y G_\varep (x,y)$
and its applications}

In this subsection we derive an asymptotic expansion for $\nabla_x\nabla_y G_\varep (x,y)$.
As its applications we obtain asymptotic expansions for
$({\partial}/{\partial x_i}) \big(\mathcal{L}_\varep\big)^{-1}({\partial}/
{\partial x_j})$ and the Dirichlet-to-Neumann map associated with $\mathcal{L}_\varep$.

\begin{thm}\label{theorem-3.3.1}
Suppose that $A(y)$ satisfies conditions (\ref{ellipticity}), (\ref{periodicity})
and (\ref{smoothness}). Let $\Omega$ be a $C^{3,\eta}$ domain for some
$\eta>0$.
Then
\begin{equation}\label{estimate-3.3.1}
\aligned
\big| \frac{\partial^2}{\partial x_i\partial y_j}\big\{ G_\varep^{\alpha\beta} (x,y)\big\}
& -\frac{\partial}{\partial x_i}
\big\{ \Phi_{\varep, k}^{\alpha\gamma} (x)\big\}\cdot
 \frac{\partial^2}{\partial x_k\partial y_\ell}\big\{ G_0^{\gamma \sigma} (x,y)\big\}
\cdot
\frac{\partial}{\partial y_j} \big\{ \Phi_{\varep, \ell}^{*\beta \sigma} (y)\big\}
\big|\\
&\le \frac{C\varep \ln \big[ \varep^{-1} |x-y|+2\big]}
{|x-y|^{d+1}}
\endaligned
\end{equation}
for any $x,y\in \Omega$, where $C$ depends only on $d$, $m$, $\mu$, $\lambda$, $\tau$
and $\Omega$.
\end{thm}

\begin{proof}
Fix $x_0, y_0\in \Omega$.
Let $r=|x_0-y_0|/8$.
Since $|\nabla_x\nabla_y G_\varep (x,y)|\le C |x-y|^{-d}$, it suffices to consider the case
$\varep<r$. Fix $1\le \beta\le m$ and $1\le j\le d$, let
$$
\left\{
\aligned
u^\alpha_\varep (x) & =\frac{\partial G_\varep^{\alpha\beta}}{\partial y_j}  (x, y_0),\\
u_0^\alpha (x) & = \frac{\partial}{\partial y_j} \big\{ \Phi_{\varep, \ell}^{*\beta \sigma}\big\} (y_0)
\cdot \frac{\partial  G_0^{\alpha \sigma}}{\partial y_\ell} (x, y_0)
\endaligned
\right.
$$
in $ D_{4r}=D(x_0, 4r)$.
In view of (\ref{adjoint-estimate-2}) we obtain
\begin{equation}\label{3.3.2}
\| u_\varep -u_0\|_{L^\infty (D_{4r})}
\le {C \varep r^{-d} \ln \big[\varep^{-1} r +2\big]}.
\end{equation}
Since $\Omega$ is $C^{3,\eta}$, we have $\|\nabla u_0\|_{L^\infty(D_{4r})} \le Cr^{-d}$, 
\begin{equation}\label{3.3.3}
\|\nabla^2 u_0\|_{L^\infty (D_{4r})} \le Cr^{-d-1}
\quad \text{ and }\quad 
\|\nabla^2 u_0\|_{C^{0,\eta}(D_{4r})} \le Cr^{-d-1-\eta}.
\end{equation}
By Lemma \ref{lemma-3.2.1}, estimates (\ref{3.3.2}) and (\ref{3.3.3}) imply that
$$
\|\frac{\partial u_\varep^\alpha}{\partial x_i}
-\frac{\partial}{\partial x_i} \big\{ \Phi_{\varep, k}^{\alpha\gamma}\big\}
\cdot \frac{\partial u_0^\gamma}{\partial x_k}
\|_{L^\infty (D_r)}
\le \frac{C\varep \ln \big[\varep^{-1} r+2\big]}{r^{d+1}}.
$$
This gives the desired estimate (\ref{estimate-3.3.1}).
\end{proof}

Let $u_\varep\in H_0^1(\Omega)$ and $\mathcal{L}_\varep (u_\varep)=\text{div}(f)$ in $\Omega$,
where $f=(f_j^\beta)\in L^2(\Omega)$.
Then $\frac{\partial u_\varep^\alpha}{\partial x_i} =T_{\varep, ij}^{\alpha\beta} ( f_j^\beta)$,
where
$$
T_{\varep, ij}^{\alpha\beta}(g) (x)
=\text{\rm p.v.}
\int_\Omega \frac{\partial^2}{\partial x_i\partial y_j}
\big\{ G_\varep^{\alpha\beta} (x,y)\big\} g (y)\, dy.
$$
It is known that if $\Omega$ is $C^{1,\eta}$ and $A(y)$ satisfies
conditions (\ref{ellipticity})-(\ref{smoothness}),
operators $T_{\varep, ij}^{\alpha\beta}$ are uniformly
bounded on $L^p(\Omega)$ for $1<p<\infty$, and of weak type $(1,1)$ \cite{AL-1991}.

\begin{thm}\label{theorem-3.3.2}
Suppose that $\Omega$ and $A(y)$ satisfy the same conditions as in Theorem \ref{theorem-3.3.1}.
Let $1<p<\infty$ and $g\in L^p(\Omega)$. Then as $\varep\to 0$,
\begin{equation}\label{estimate-3.3.2}
\aligned
T_{\varep, ij}^{\alpha\beta} (g)
-\frac{\partial}{\partial x_i} \big\{ \Phi_{\varep, k}^{\alpha\gamma}\big\}
\cdot T_{0, k\ell}^{\gamma \sigma}
\left( \frac{\partial}{\partial x_j}
\big\{ \Phi_{\varep, \ell}^{*\beta \sigma}\big\} g\right)
& +\frac{\partial}{\partial x_i} \big\{ \Phi_{\varep, k}^{\alpha\gamma}\big\}
\cdot T_{0, k\ell}^{\gamma \sigma}
\left( \frac{\partial}{\partial x_j}
\big\{ \Phi_{\varep, \ell}^{*\beta \sigma}\big\} \right) \cdot g\\
\to 0 \quad \text{ in } L^q(\Omega),
\endaligned
\end{equation}
for any $1<q<p$.
\end{thm}

\begin{proof}
Let $S_\varep (g)$ denote the left hand side of (\ref{estimate-3.3.2}).
By uniform boundedness of $\|T^{\alpha\beta}_{\varep, ij}\|_{L^p\to L^p}$,
$\|\nabla\Phi_{\varep}\|_\infty$ and $\|\nabla\Phi^*_\varep\|_\infty$,
we see that the operators $S_\varep: L^p(\Omega)\to L^q(\Omega)$ are
uniformly bounded for $1<q<p< \infty$. As a result
we may assume that $g\in C^1(\br^d)$.

Note that $S_\varep (1)=0$ and
$$
S_\varep (g) (x) = S_\varep (g -g(x)) (x)
=\text{\rm p.v.}
\int_\Omega K_\varep (x,y) \big\{ g(y)-g(x)\big\} \, dy,
$$
where, by Theorem \ref{theorem-3.3.1} and the estimate 
$|\nabla_x\nabla_y G_\varep (x,y)|\le C |x-y|^{-d}$, 
the integral kernel $K_\varep(x,y)$ satisfies
$$
|K_\varep (x,y)|
\le C \min \left\{\frac{1}{|x-y|^d},
 \frac{ \varep \ln [\varep^{-1} |x-y|+2]}{|x-y|^{d+1}}\right\}.
$$
It follows that if $\varep <(1/2)$,
$$
\aligned
|S_\varep (g)(x)| & \le C \| g\|_{C^1(\Omega)} \int_{B(x,\sqrt{\varep})} 
\frac{dy}{|x-y|^{d-1}}
+ C \varep \| g\|_{L^\infty(\Omega)}
\int_{\Omega\setminus B(x,\sqrt{\varep})}
\frac{\ln [\varep^{-1}|x-y| +2]}{|x-y|^{d+1}}\, dy\\
&\le C \sqrt{\varep}|\ln {\varep}| \| g\|_{C^1(\Omega)}.
\endaligned
$$
Hence, $\| S_\varep (g)\|_{L^q(\Omega)} \to 0$ as $\varep\to 0$.
This completes the proof.
\end{proof}

Finally, we consider the Dirichlet-to-Neumann map $\Lambda_\varep$
associated with the operator $\mathcal{L}_\varep$.
Let $f\in H^{1/2}(\partial\Omega)$ and $u_\varep\in H^1(\Omega)$
be the solution of $\mathcal{L}_\varep (u_\varep)=0$ in $\Omega$
and $u_\varep =f$ on $\partial\Omega$.
The map $\Lambda_\varep: H^{1/2}(\partial\Omega)
\to H^{-1/2}(\partial\Omega)$ is defined by
$\Lambda_\varep (f) =\frac{\partial u_\varep}{\partial\nu_\varep}$.
It is known that
$\Lambda_\varep: W^{1,p}(\partial\Omega) \to L^p(\partial\Omega)$
is uniformly bounded for $1<p<\infty$, if $\Omega$ is $C^{1, \eta}$  \cite{KLS1}.
In the case that $\Omega$ is Lipschitz, the map is
uniformly bounded for $1<p<2+\delta$ if  $m=1$ \cite{Kenig-Shen-1}, and
for $p$ close to $2$ if $m>1$ \cite{Kenig-Shen-2}.
For simplicity we will assume $m=1$ and $A^*=A$
in the rest of this subsection.

Let 
$$
 \omega_\varep (x)= n_i(x) n_j (x) a_{ij}(x/\varep)\cdot 
[n_k(x) n_\ell (x) \hat{a}_{k\ell}]^{-1}
\cdot \frac{\partial}{\partial n} \big\{
\Phi_{\varep, s} \big\} \cdot n_s(x).
$$
Then $\|\omega_\varep\|_{L^\infty(\partial\Omega)} \le C$.
It follows from Theorem \ref{theorem-3.3.1} that
\begin{equation}\label{3.3.10}
\aligned
\big|
\frac{\partial}{\partial \nu_\varep (x)}
\big\{ P_\varep (x,y) \big\}
-& \omega_\varep (x) \cdot
\frac{\partial}{\partial \nu_0 (x)}
 \big\{ P_0 (x,y) \big\}
\cdot \omega_\varep (y)\big|\\
& \le 
\frac{C \varep \ln [\varep^{-1}|x-y|+2]}{|x-y|^{d+1}}
\endaligned
\end{equation}
for any $x,y\in \partial\Omega$.
Using $\Phi_{\varep, k} (x)=x_k$ on $\partial\Omega$, 
one may show that $n_in_j \hat{a}_{ij}\omega_\varep (x)=n_k\Lambda_\varep (x_k)$.

\begin{thm}\label{theorem-3.3.3}
Suppose that $\Omega$ and $A(y)$ satisfy the same conditions as in Theorem \ref{theorem-3.3.1}.
We further assume that $m=1$ and $A^*=A$.
Let $f\in H^1(\partial\Omega)$.
Then
\begin{equation}\label{estimate-3.3.3}
\aligned
\Lambda_\varep (f)-n_i\frac{\partial f}{\partial t_{ij}} \Lambda_\varep (x_j)
&  +\omega_\varep \big[  f\Lambda_0 (\omega_\varep) 
- \Lambda_0 (\omega_\varep f)\big]\\
&+ \omega_\varep n_i \frac{\partial f}{\partial t_{ij}}
\big[ \Lambda_0 (\omega_\varep x_j)
- x_j \Lambda_0 (\omega_\varep)\big]
\to 0,
\endaligned
\end{equation}
in $L^q(\partial\Omega)$ for any $1<q<2$, 
where $\frac{\partial f}{\partial t_{ij}}
=(n_i \frac{\partial}{\partial x_j} -n_j \frac{\partial}{\partial x_i})f$.
\end{thm}

\begin{proof} By a linear change of variables 
Theorems \ref{theorem-5.1} and \ref{theorem-5.2} in Section 5 continue to hold for
$\Lambda=\Lambda_0$. It follows that
 $\|\Lambda_0 (\omega_\varep f)-f\Lambda_0 (\omega_\varep)\|_{L^2(\partial\Omega)}
\le C \| f\|_{H^1(\partial\Omega)}$
and $\|\Lambda_0 (\omega_\varep x_j)-x_j \Lambda_0 (\omega_\varep)\|_{L^p(\partial\Omega)}
\le C_p \| \omega_\varep\|_{L^p(\partial\Omega)} \le C_p$ for any $1<p<\infty$.
Hence, by H\"older's inequality,
the left hand side of (\ref{estimate-3.3.3}),
as an operator, is uniformly bounded from
$H^1(\partial\Omega)$ to $L^q(\partial\Omega)$ for any $1<q<2$.
Consequently, it suffices to show that for $f\in C^2(\partial\Omega)$,
the left hand side of (\ref{estimate-3.3.3}) goes to zero in 
$L^\infty(\partial\Omega)$, as $\varep\to 0$.

To this end, recall that $\Lambda_\varep (f)= \frac{\partial u_\varep}{\partial\nu_\varep}$,
where $u_\varep (z)=\int_\Omega P_\varep (z,y) f(y)\, d\sigma (y)$ for
$z\in \Omega$.
Write
\begin{equation}\label{3.3.12}
\aligned
u_\varep (z) -f(x)
&=\int_{\partial\Omega} P_\varep (z,y) \big\{ f(y)-f(x)\big\}\, d\sigma (y)\\
&=\int_{\partial\Omega} P_\varep (z,y)
\left\{ f(y)-f(x) -n_i(x) \frac{\partial f}{\partial t_{ij}(x)} \cdot (y_j-x_j)\right\}\, d\sigma (y)\\
&\qquad
+\int_{\partial\Omega}
P_\varep (z,y) n_i (x) \frac{\partial f}{\partial t_{ij}(x)} \cdot (y_j-x_j)\, d\sigma (y)\\
&=\int_{\partial\Omega} P_\varep (z,y)
\left\{ f(y)-f(x) -n_i(x) \frac{\partial f}{\partial t_{ij}(x)} \cdot (y_j-x_j)\right\}\, d\sigma (y)\\
&\qquad
+n_i (x) \frac{\partial f}{\partial t_{ij}(x)} \cdot \big[ \Phi_{\varep, j} (z)-x_j\big].
\endaligned
\end{equation}
Since $|f(y)-f(x)-n_i(x)\cdot \partial f/\partial t_{ij}(x) \cdot (y_j -x_j)|
\le C_f |y-x|^2$ for $x,y\in \partial\Omega$,
it follows by taking derivatives in $z$ and then letting $z\to x$ in (\ref{3.3.12}) that
$$
\aligned
\Lambda_\varep (f) (x)
&= 
\int_{\partial\Omega} \frac{\partial}{\partial \nu_\varep (x)}
\big\{ P_\varep (x,y)\big\}
\left\{ f(y)-f(x) -n_i(x) \frac{\partial f}{\partial t_{ij}(x)} \cdot (y_j-x_j)\right\}\, d\sigma (y)\\
&\qquad
+n_i (x) \frac{\partial f}{\partial t_{ij}(x)} \Lambda_\varep (x_j).
\endaligned
$$
In view of (\ref{3.3.10}) as well as the estimate
$|\nabla_x P_\varep(x,y)|\le C |\nabla_x\nabla_y G_\varep (x,y)|
\le C|x-y|^{-d}$, we obtain
$$
\aligned
& \Lambda_\varep (f) (x) =I_\varep (x)
+n_i  \frac{\partial f}{\partial t_{ij}} \Lambda_\varep (x_j)\\
&+
\int_{\partial\Omega} \omega_\varep (x) \frac{\partial}{\partial \nu_0 (x)}
\big\{ P_0 (x,y)\big\} \omega_\varep (y)
\left\{ f(y)-f(x) -n_i(x) \frac{\partial f}{\partial t_{ij}(x)} \cdot (y_j-x_j)\right\}\, d\sigma (y)\\
&=
I_\varep (x)
+n_i  \frac{\partial f}{\partial t_{ij}} \Lambda_\varep (x_j)
+\omega_\varep \big[\Lambda_0 (\omega_0 f)-f \Lambda_0 (\omega_\varep)\big]
+\omega_\varep n_i \frac{\partial f}{\partial t_{ij}}
\big[ x_j \Lambda_0 (\omega_\varep)
-\Lambda_0 (\omega_\varep x_j)\big],
\endaligned
$$
where the term $I_\varep (x)$ satisfies
$$
\aligned
|I_\varep (x)|
&\le C_f
\left\{  \varep \int_{\partial\Omega\setminus B(x, \varep)}
\frac{\ln [\varep^{-1} |x-y|+2]}{|y-x|^{d-1}}\, d\sigma(y)
+ \int_{B(x, \varep)\cap \partial\Omega} \frac{d\sigma(y)}{|y-x|^{d-2}}\right\}\\
&\le C_f\, \varep [\ln (\varep^{-1}M+2)]^2.
\endaligned
$$
This gives the desired estimate.
\end{proof}

\section{Asymptotic behavior of Neumann functions}

Throughout this section we will assume that $A(y)$ satisfies conditions
(\ref{ellipticity}), (\ref{periodicity}) and (\ref{smoothness}).
Under these conditions one may construct a matrix of Neumann functions
$N_\varep (x,y)=\big( N_\varep^{\alpha\beta} (x,y)\big)$ in a bounded 
Lipschitz domain $\Omega$ such that 
\begin{equation}\label{definition-of-Neumann functions}
\left\{
\aligned
\mathcal{L}_\varep \big\{ N_\varep^\beta (\cdot,y)\big\} & =e^\beta \delta_y(x) \quad \text{ in }
\Omega,\\
\frac{\partial}{\partial \nu_\varep} \big\{ N_\varep^\beta (\cdot, y)\big\}
& =-e^\beta |\partial\Omega|^{-1}\quad \text{ on } \partial\Omega,\\
\int_{\partial\Omega} N_\varep^\beta (x,y)d\sigma (x) & =0,
\endaligned
\right.
\end{equation}
where $N_\varep^\beta (x,y) =(N_\varep^{1\beta} (x,y), \dots, N_\varep^{m\beta}(x,y))$ and
$e^\beta=(0, \dots, 1, \dots, 0)$ with $1$ in the $\beta^{th}$ position.
Let $u_\varep \in H^1(\Omega)$ be a solution to
$\mathcal{L}_\varep (u_\varep)=F$ in $\Omega$ and $\frac{\partial u_\varep}{\partial\nu_\varep}
=g$ on $\partial\Omega$. Then
\begin{equation}\label{Neumann-representation}
u_\varep (x) -\frac{1}{|\partial\Omega|}
\int_{\partial\Omega} u_\varep
=\int_{\Omega} N_\varep (x,y) F(y)\, dy
+\int_{\partial\Omega} N_\varep (x,y) g(y)\, d\sigma(y).
\end{equation}
Under the additional assumption that $A^*=A$ and $\Omega$ is $C^{1,\eta}$
for some $\eta\in (0,1)$, 
it was proved in \cite{KLS1} that
\begin{equation}\label{Neumann-estimate-4.0.1}
\left\{
\aligned
|N_\varep (x,y)| & \le \frac{C}{|x-y|^{d-2}},\\
|\nabla_x N(x,y)|+|\nabla_y N_\varep (x,y)| &\le \frac{C}{|x-y|^{d-1}},\\
|\nabla_x\nabla_y N_\varep (x,y)|& \le \frac{C}{|x-y|^d}
\endaligned
\right.
\end{equation}
for any $x,y\in \Omega$.
The goal of this section is to establish  the asymptotic estimates
of $N_\varep(x,y)$ and $\nabla_x N_\varep (x,y)$
in Theorem \ref{theorem-B}.

\subsection{$L^\infty$ estimates}

The goal of this subsection is to prove the estimate (\ref{Neumann's-size}).
We also obtain an $O(\varep)$ estimate (up to a logarithmic factor) for
$\|u_\varep-u_0\|_{L^p(\Omega)}$ for solutions with Neumann conditions.
Recall that $D_r =D_r(x_0, r)=B(x_0,r)\cap\Omega$ and
$\Delta_r =\Delta(x_0,r)=B(x_0, r)\cap \partial\Omega$
for some $x_0\in \overline{\Omega}$ and $0<r<r_0$.
We begin with an $L^\infty$ estimate for local solutions.

\begin{lemma}\label{lemma-4.1.1}
Let $\Omega$ be a bounded $C^{1,\eta}$ domain for some $\eta\in (0,1)$.
Let $u_\varep\in H^1(D_{3r})$ and $ u_0\in W^{2,p}(D_{3r})$ for some $p>d$.
Suppose that $\mathcal{L}_\varep (u_\varep)=\mathcal{L}_0(u_0)$ in $D_{3r}$
and $\frac{\partial u_\varep}{\partial\nu_\varep}=\frac{\partial u_0}{\partial\nu_0}$
on $\Delta_{3r}$.
Then,  if $0<\varep<(r/2)$,
\begin{equation}\label{estimate-4.1.1}
\aligned
\| u_\varep -u_0\|_{L^\infty(D_r)}
&\le \frac{C}{r^d} \int_{D_{3r}} |u_\varep -u_0|
+C\varep \ln [\varep^{-1}r+2] \|\nabla u_0\|_{L^\infty(D_{3r})}\\
&\qquad\qquad
+C_p\, \varep r^{1-\frac{d}{p}} \|\nabla^2 u_0\|_{L^p(D_{3r})}.
\endaligned
\end{equation}
\end{lemma}

\begin{proof}
By rescaling we may assume $r=1$.
Choose a $C^{1,\eta}$ domain $\widetilde{D}$ such that $D_2\subset \widetilde{D}\subset{D}_3$.
Let 
$$
w_\varep =u_\varep (x)-u_0(x)- \varep \chi_j^\beta (x/\varep) \frac{\partial u_0^\beta}
{\partial x_j}.
$$
Using Proposition \ref{prop-2.1} with 
$V_{\varep, j}^\beta =P_j^\beta (x)+\varep \chi_j^\beta (x/\varep)$, we see that
\begin{equation}\label{4.1.1-0}
\left(\mathcal{L}(w_\varep) \right)^\alpha =\varep 
\frac{\partial}{\partial x_i} \left\{ b_{ijk}^{\alpha\beta} (x/\varep)
\frac{\partial^2 u_0^\beta}{\partial x_j\partial x_k}\right\},
\end{equation}
where $b_{ijk}^{\alpha \beta} (y)
=F_{jik}^{\alpha\beta} (y)+a_{ij}^{\alpha\gamma} (y) \chi_k^{\gamma\beta}(y)$
is a bounded periodic function.
Also, by a direct computation, we have
\begin{equation}\label{4.1.1-1}
\aligned
\left(\frac{\partial w_\varep}{\partial \nu_\varep}\right)^\alpha
= & \left(\frac{\partial u_\varep}{\partial \nu_\varep}\right)^\alpha
-\left(\frac{\partial u_0}{\partial \nu_0}\right)^\alpha
+
\frac{\varep}{2}
\left( n_i\frac{\partial}{\partial x_j} -n_j\frac{\partial}{\partial x_i}\right)
\left\{ F_{jik}^{\alpha\gamma}(x/\varep) \frac{\partial u_0^\gamma}
{\partial x_k}\right\}\\
& \qquad
-\varep n_i b_{ijk}^{\alpha\beta} (x/\varep)
\frac{\partial^2 u_0^\beta}{\partial x_j\partial x_k}
\endaligned
\end{equation}
(see \cite[Lemma 5.1]{KLS2}). 
Let $w_\varep =w_\varep^{(1)}  +w_\varep^{(2)}$, where
\begin{equation}\label{4.1.1-2}
\big( w_\varep^{(1)} (x)\big)^\alpha
=-\varep \int_{\wD} \frac{\partial}{\partial y_i}
\big\{ \wN^{\alpha\beta}_\varep (x,y)\big\}\cdot
b_{ijk}^{\beta\gamma} (y/\varep) \cdot \frac{\partial^2 u_0^\gamma}
{\partial y_j\partial y_k} \, dy
\end{equation}
and $\wN_\varep (x,y)$
denotes the matrix of Neumann functions for $\mathcal{L}_\varep$
in $\wD$.
Since $|\nabla_y \wN_\varep (x,y)|\le C |x-y|^{1-d}$, it follows from H\"older's
inequality that 
\begin{equation}\label{4.1.1}
\|w_\varep^{(1)}\|_{L^\infty (D_2)} \le C_p \, \varep \|\nabla^2 u_0\|_{L^p(D_3)}
\qquad \text{ for any } p>d.
\end{equation}

To estimate $w_\varep^{(2)}$, we observe that
$\mathcal{L}_\varep (w_\varep^{(2)}) =0$ in $\wD$ and
\begin{equation}\label{4.1.1-3}
\left(\frac{\partial w^{(2)}_\varep}{\partial \nu_\varep}\right)^\alpha
= \left(\frac{\partial u_\varep}{\partial \nu_\varep}\right)^\alpha
-\left(\frac{\partial u_0}{\partial \nu_0}\right)^\alpha
+
\frac{\varep}{2}
\left( n_i\frac{\partial}{\partial x_j} -n_j\frac{\partial}{\partial x_i}\right)
\left\{ F_{jik}^{\alpha\gamma}(x/\varep) \frac{\partial u_0^\gamma}
{\partial x_k}\right\}
\end{equation}
on $\partial\wD$. Let $w_\varep^{(2)} =w_\varep^{(21)} +w_\varep^{(22)}$, where
\begin{equation}\label{4.1.1-4}
\big( w_\varep^{(21)}\big)^\alpha (x)
=-\frac{\varep}{2}
\int_{\partial\wD} 
\left( n_i\frac{\partial}{\partial y_j} -n_j\frac{\partial}{\partial y_i}\right)
\big\{ \wN^{\alpha\beta}_\varep (x,y)\big\}
\cdot F_{jik}^{\beta\gamma}(y/\varep) \frac{\partial u_0^\gamma}
{\partial y_k}\, d\sigma (y).
\end{equation}
For each $x\in \wD$, choose 
$\hat{x}\in\partial \wD$ such that $|\hat{x}-x|=\text{dist}(x, \partial\wD)$.
Note that for $y\in \partial\wD$, $|y-\hat{x}|
\le |y-x| +|x-\hat{x}|\le 2|y-x|$. Thus,
$|\nabla_y \wN_\varep (x,y)|\le C |y-\hat{x}|^{1-d}$ and
$$
\aligned
|\big( w_\varep^{(21)}\big)^\alpha (x)|
& = \frac{\varep}{2}
\big|\int_{\partial\wD} 
\left( n_i\frac{\partial}{\partial y_j} -n_j\frac{\partial}{\partial y_i}\right)
\big\{ \wN^{\alpha\beta}_\varep (x,y)\big\}
\cdot \big\{ f_{ji}^\beta (y)-f_{ji}^\beta (\hat{x})\big\} \, d\sigma (y)\big|\\
&
\le C\varep \int_{\partial\wD}
\frac{|f(y)-f(\hat{x})|}{|y-\hat{x}|^{d-1}}\, d\sigma(y),
\endaligned
$$
where $f(y)=(f_{ji}^\beta (y)) =\big(F_{jik}^{\beta\gamma}(y/\varep) \frac{\partial u_0^\gamma}
{\partial y_k} (y)\big)$.
Since $\| f\|_{L^\infty(D_3)} \le C \|\nabla u_0\|_{L^\infty(D_3)}$ and
$$
|f(y)-f(\hat{x})|
\le C\varep^{-1} |y-\hat{x}| \|\nabla u_0\|_{L^\infty(D_3)}
+|y-\hat{x}|^\rho \|\nabla u_0\|_{C^{0,\rho}(D_3)},
$$
where $0<\rho<\eta$ and we have used the fact $\|F_{jik}^{\beta\gamma}\|_{C^1 (Y)}\le C$, it follows that
$$
\aligned
|w_\varep^{(21)} (x)|
& \le C\varep \| \nabla u_0\|_{L^\infty(D_3)}
\int_{\partial\wD\setminus B(\hat{x}, \varep)} |\hat{x}-y|^{1-d} \, d\sigma(y)\\
&\qquad + C\|\nabla u_0\|_{L^\infty(D_3)}\int_{B(\hat{x}, \varep)\cap \partial\wD} 
|y-\hat{x}|^{2-d}\, d\sigma (y)\\
&\qquad +C\varep \|\nabla u_0\|_{C^{0,\rho}(D_3)}
\int_{B(\hat{x}, \varep)\cap \partial\wD} 
 |y-\hat{x}|^{1-d+\rho}\, d\sigma (y)\\
& \le C\varep \ln[\varep^{-1}+2] \|\nabla u_0\|_{L^\infty(D_3)}
+C\varep^{1+\rho} \| \nabla u_0\|_{C^{0,\rho}(D_3)}.
\endaligned
$$ 
By Sobolev imbedding, this implies that
\begin{equation}\label{4.1.2}
\|w_\varep^{(21)}\|_{L^\infty(D_2)}
 \le C\varep \ln[\varep^{-1}+2] \|\nabla u_0\|_{L^\infty(D_3)}
+ C_p\, \varep \|\nabla^2 u_0\|_{L^p(D_3)}
\end{equation}
for any $p>d$.

Finally, to estimate $w_\varep^{(22)}$, we note that
$\mathcal{L}_\varep (w_\varep^{(22)}) =0$ in $D_2$ and
$\frac{\partial}{\partial\nu_\varep} \{ w_\varep^{(22)}\}=
\frac{\partial u_\varep}{\partial\nu_\varep} -
\frac{\partial u_0}{\partial\nu_0} =0$ on $\Delta_2$.
It follows from \cite[Theorem 3.1]{KLS1} that
$$
\aligned
\| w_\varep^{(22)}\|_{L^\infty(D_1)}
&\le C \int_{D_2} |w_\varep^{(22)}|\, dx\\
& \le C \int_{D_2} |u_\varep-u_0|\, dx
+C\varep \|\nabla u_0\|_{L^\infty(D_2)}\\
&\qquad\qquad
+C \|w_\varep^{(1)}\|_{L^\infty(D_2)}
+C \|w_\varep^{(21)}\|_{L^\infty(D_2)}.
\endaligned
$$
Hence,
$$
\aligned
\|u_\varep -u_0\|_{L^\infty(D_1)}
\le & C \int_{D_2} |u_\varep-u_0|\, dx
+C\varep \|\nabla u_0\|_{L^\infty(D_2)}\\
&\qquad +C \|w_\varep^{(1)}\|_{L^\infty(D_2)}
+C \|w_\varep^{(21)}\|_{L^\infty(D_2)}.
\endaligned
$$
This, together with (\ref{4.1.1}) and (\ref{4.1.2}), gives (\ref{estimate-4.1.1}).
\end{proof}

The next lemma on the traces of fractional integrals is known.
We provide a proof for the sake of completness.

\begin{lemma}\label{lemma-4.1.2}
Let $\Omega$ be a bounded Lipschitz domain.
Let $g\in C_0^1 (\Omega\cap B(y_0,r))$ for some $y_0\in \overline{\Omega}$ and
$$
u(x)=\int_\Omega \frac{g(y)\, dy}{|x-y|^{d-1}}.
$$
Then $\|u\|_{L^2(\partial\Omega)} \le Cr^{1/2} \| g\|_{L^2(\Omega)}$.
\end{lemma}

\begin{proof}
By the well known estimates for fractional and singular integrals,
$\|\nabla u\|_{L^2(\Omega)}
\le C \| g\|_{L^2(\Omega)}$ and $\| u\|_{L^2(\Omega)} \le C \| g\|_{L^p(\Omega)}$,
where $p=\frac{2d}{d+2}$. Choose a vector field $v\in C_0^\infty (\br^d, \br^d)$
such that
$<v, n>\ge c_0>0$ on $\partial\Omega$ and $|v|\le 1$, $|\nabla v|\le C/r$
in $\br^d$.
It follows from integration by parts and H\"older's inequality that
$$
\aligned
c_0\int_{\partial \Omega}|u|^2\, d\sigma
& \le \int_{\partial\Omega} <v,n>|u|^2\, d\sigma
\le Cr^{-1}\int_\Omega |u|^2\, dx
+C \int_{\Omega} |u||\nabla u|\, dx\\
& \le Cr^{-1} \int_\Omega |u|^2\, dx +Cr \int_\Omega |\nabla u|^2\, dx\\
&\le Cr^{-1} \| g\|^2_{L^p(\Omega)} + Cr \|g\|_{L^2(\Omega)}^2\\
&\le Cr \| g\|^2_{L^2(\Omega)},
\endaligned
$$
where we have used the fact supp$(g)\subset B(y_0,r)$.
This completes the proof.
\end{proof}

We are now ready to prove the estimate (\ref{Neumann's-size}).

\begin{thm}\label{theorem-4.1.2}
Suppose that $A(y)$ satisfies conditions (\ref{ellipticity}),
(\ref{periodicity}), (\ref{smoothness}) and $A^*=A$.
Let $\Omega$ be a bounded $C^{1,1}$ domain.
Then
\begin{equation}\label{estimate-4.1.2}
|N_\varep (x,y)-N_0 (x,y)|\le
\frac{C\varep \ln [\varep^{-1}|x-y|+2]}{|x-y|^{d-1}},
\end{equation}
for any $x, y\in \Omega$, where $C$ depends only on $d$, $m$, $\mu$, $\lambda$,
$\tau$ and $\Omega$.
\end{thm}

\begin{proof}
By rescaling we may assume that diam$(\Omega)= 1$.
Fix $x_0,y_0\in \Omega$ and let $r=|x_0-y_0|/8$. 
Since $|N_\varep (x_0,y_0)|\le Cr^{2-d}$, we may assume that
$\varep<r$.
For $g\in C_0^\infty(D(y_0,r))$ and $\varep\ge 0$,
let
$
u_\varep (x)=\int_\Omega N_\varep (x,y) g(y)\, dy.
$
Then $\mathcal{L}_\varep (u_\varep)=g$ in $\Omega$,
$\frac{\partial u_\varep}{\partial\nu_\varep}
=-\frac{1}{|\partial\Omega|}\int_\Omega g$ on $\partial\Omega$
and $\int_{\partial\Omega} u_\varep =0$.
It follows that $\mathcal{L}_\varep (u_\varep)=\mathcal{L}_0 (u_0)$
in $\Omega$ and $\frac{\partial u_\varep}{\partial\nu_\varep}=
\frac{\partial u_0}{\partial\nu_0}$ on $\partial\Omega$.
Let
$
w_\varep =u_\varep (x) -u_0(x) -\varep \chi_j^\beta (x/\varep)
\frac{\partial u_0^\beta}{\partial x_j}.
$
As in the proof of Lemma \ref{lemma-4.1.1},
$$
\left(\mathcal{L}(w_\varep) \right)^\alpha =\varep 
\frac{\partial}{\partial x_i} \left\{ b_{ijk}^{\alpha\beta} (x/\varep)
\frac{\partial^2 u_0^\beta}{\partial x_j\partial x_k}\right\} \qquad \text{ in }\Omega
$$
and
$$
\left(\frac{\partial w_\varep}{\partial \nu_\varep}\right)^\alpha
=\frac{\varep}{2}
\left( n_i\frac{\partial}{\partial x_j} -n_j\frac{\partial}{\partial x_i}\right)
\left\{ F_{jik}^{\alpha\gamma}(x/\varep) \frac{\partial u_0^\gamma}
{\partial x_k}\right\}
-\varep n_i b_{ijk}^{\alpha\beta} (x/\varep)
\frac{\partial^2 u_0^\beta}{\partial x_j\partial x_k}
$$
on $\partial\Omega$.
Now, write $w_\varep =\theta_\varep +z_\varep +\rho$, where
$\mathcal{L}_\varep (z_\varep)=\mathcal{L}_\varep (w_\varep)$ in $\Omega$,
$\int_\Omega z_\varep=0$,
$$
\frac{\partial z_\varep}{\partial\nu_\varep}
=-\varep n_i b_{ijk}^{\alpha\beta} (x/\varep)
\frac{\partial^2 u_0^\beta}{\partial x_j\partial x_k}
\qquad \text{ on } \partial\Omega,
$$
and $\rho=\frac{1}{|\partial\Omega|} \int_{\partial\Omega} \big\{ w_\varep -z_\varep\big\}$ 
is a constant.
Note that $\|\nabla z_\varep\|_{L^2(\Omega)} \le C \varep \| \nabla^2 u_0\|_{L^2(\Omega)}
\le C\varep \| g\|_{L^2(\Omega)}$.
Since $\int_\Omega z_\varep =0$, by the Poincar\'e inequality, we obtain
$\| z_\varep\|_{L^p(\Omega)} \le C\varep \| g\|_{L^2(\Omega)}$,
where $p=\frac{2d}{d-2}$.
It follows by H\"older's inequality that
\begin{equation}\label{4.1.20}
\| z_\varep\|_{L^2(D(x_0,r))}
\le Cr^{\frac{d}{2}-\frac{d}{p}} \| z_\varep\|_{L^p(D(x_0, r))}
\le C\varep r\| g\|_{L^2(\Omega)}.
\end{equation}

Next, to estimate $\theta_\varep$, we observe that $\mathcal{L}_\varep (\theta_\varep)
=0$ in $\Omega$, $\int_{\partial\Omega} \theta_\varep =0$ and
$$
\frac{\partial \theta_\varep}{\partial \nu_\varep}
=\frac{\varep}{2}
\left( n_i\frac{\partial}{\partial x_j} -n_j\frac{\partial}{\partial x_i}\right)
\left\{ F_{jik}^{\alpha\gamma}(x/\varep) \frac{\partial u_0^\gamma}
{\partial x_k}\right\}.
$$
Using a duality argument and $L^2$ estimates for the Neumann problem in \cite{Kenig-Shen-1},
we may deduce that $\|\theta_\varep \|_{L^2(\partial\Omega)}
\le C\varep \| \nabla u_0\|_{L^2(\partial\Omega)}$ 
(see the proof of Theorem 5.2 in \cite{KLS2}).
By the square function estimate for the $L^2$ Dirichlet problem \cite{KLS2},
this implies that
$$
\|\theta_\varep\|_{H^{1/2}(\Omega)} \le C\| \theta_\varep\|_{L^2(\partial\Omega)}
\le C\varep \| \nabla u_0\|_{L^2(\partial\Omega)}.
$$
It follows from the Sobolev imbedding that $\|\theta_\varep\|_{L^{p_1}(\Omega)}
\le C\varep \| \nabla u_0\|_{L^2(\partial\Omega)}$, where $p_1=\frac{2d}{d-1}$. 
By H\"older's inequality, this
gives
\begin{equation}\label{4.1.21}
\| \theta_\varep\|_{L^2(D(x_0,r))}
\le C\varep r^{\frac12}\|\nabla u_0\|_{L^2(\partial\Omega)}.
\end{equation}
Since 
$$
|\nabla u_0 (x)|\le C \int_\Omega \frac{|g(y)|\, dy}{|x-y|^{d-1}},
$$
we may invoke Lemma \ref{lemma-4.1.2} to claim  that $\|\nabla u_0\|_{L^2(\Omega)}
\le C r^{1/2}\| g\|_{L^2(\Omega)}$. In view of (\ref{4.1.21})
we obtain $\| \theta_\varep\|_{L^2(D(x_0,r))} \le C\varep r \| g\|_{L^2(\Omega)}$.
This, together with (\ref{4.1.20}) and the observation
$$
|\rho|\le C \int_{\partial\Omega} 
\big\{ \varep |\nabla u_0| +|z_\varep|\big\} d\sigma \le C \varep \| g\|_{L^2(\Omega)},
$$
gives $\| w_\varep \|_{L^2(D(x_0,r))} \le C\varep r \| g\|_{L^2(\Omega)}$.
It follows that
$$
\left\{\frac{1}{r^d} \int_{D(x_0,r)} |u_\varep -u_0|^2\right\}^{1/2}
\le C \varep r^{\frac{2-d}{2}} \| g\|_{L^2(\Omega)}.
$$
Since $\|\nabla u_0\|_{L^\infty (D(x_0,r))}
\le C r^{\frac{2-d}{2}} \| g\|_{L^2(\Omega)}$
and $\|\nabla^2 u_0\|_{L^p (\Omega)} \le C \| g\|_{L^p(\Omega)}$,
by Lemma \ref{lemma-4.1.1}, we obtain
$$
|u_\varep (x_0)-u_0 (x_0)|\le C_p\, \varep r^{1-\frac{d}{p}} \ln [\varep^{-1} r+2]
\| g\|_{L^p(\Omega)},
$$
where $p>d$.
By duality this gives
$$
\left\{\frac{1}{r^d}\int_{D(y_0,r)}
|N_\varep (x_0,y)-N_0(x_0,y)|^{p^\prime}\, dy\right\}^{1/p^\prime}
\le C_p\,  \varep r^{1-d} \ln [\varep^{-1} r+2].
$$

Finally, since 
$$ 
\frac{\partial}{\partial \nu_\varep (y)}
\big\{ N_\varep (x,y)\big\}
=\frac{\partial}{\partial \nu_0 (y)}
\big\{ N_0 (x,y)\big\}=-\frac{1}{|\partial\Omega|} \qquad \text{ on }\partial\Omega,
$$
$|\nabla_y N_0 (x,y)|\le C|x-y|^{1-d}$ and
$\| \nabla^2_y N_0(x_0, y)\|_{L^p(D(y_0,r))}
\le Cr^{\frac{d}{p}-d}$,
we may invoke Lemma \ref{lemma-4.1.1} again to obtain
$$
\aligned
|N_\varep (x_0, y_0)-N_0 (x_0,y_0)|
& \le \frac{C}{r^d}\int_{D(y_0,r)}
|N_\varep (x_0,y)-N_0(x_0,y)|\, dy\\
&\qquad\qquad
+C\varep r^{1-d} \ln [\varep^{-1} r+2]\\
&\le 
C\varep r^{1-d} \ln [\varep^{-1} r+2].
\endaligned
$$
This completes the proof.
\end{proof}

\begin{remark}
{\rm
It is not clear whether the logarithmic factor in (\ref{estimate-4.1.2}) is
necessary. Also,
in view of Theorem \ref{theorem-3.1.1} on Green's functions, it would be interesting to show
that the estimate (\ref{estimate-4.1.2}) holds for scalar equations
with no smoothness condition on the coefficients.
}
\end{remark}

As a corollary of Theorem \ref{theorem-4.1.2}, we obtain an $O(\varep)$ (up to a logarithmic factor)
estimate
for $\|u_\varep -u_0\|_{L^q(\Omega)}$.

\begin{thm}\label{theorem-4.1.3}
Suppose that $A(y)$ and $\Omega$ satisfy the same assumptions 
as in Theorem \ref{theorem-4.1.2}. Let $1<p<\infty$.
For $\varep\ge 0$ and $F\in L^p(\Omega)$ with $\int_\Omega F=0$, 
let $u_\varep\in W^{1,p}(\Omega)$ be the solution to
the Neumann problem:
$\mathcal{L}_\varep (u_\varep)=F$ in $\Omega$, $\frac{\partial u_\varep}{\partial\nu_\varep}
=0$ and $\int_{\partial \Omega} u_\varep =0$.
Then
\begin{equation}\label{estimate-4.1.3}
\| u_\varep -u_0\|_{L^q (\Omega)}
\le C \varep \ln [\varep^{-1} M +2] \| F\|_{L^p(\Omega)}
\end{equation}
holds if $1<p<d$ and $\frac{1}{q}=\frac{1}{p}-\frac{1}{d}$, or $p>d$ and $q=\infty$,
where $M=\text{\rm diam} (\Omega)$.
Moreover,
\begin{equation}\label{estimate-4.1.3.1}
 \| u_\varep -u_0\|_{L^\infty (\Omega)}
\le C \varep \big[\ln (\varep^{-1} M +2)\big]^{{2-\frac{1}{d}}}
 \| F\|_{L^d(\Omega)}.
\end{equation}
\end{thm}

\begin{proof}
Note that by the estimate (\ref{estimate-4.1.2}),
$$
\aligned
|u_\varep (x)-u_0(x)|
& \le \int_{\Omega} |N_\varep (x,y)-N_0 (x,y)|| F(y)|\, dy\\
& \le C \varep \ln (\varep^{-1} M+2) \int_\Omega
\frac{|F(y)| dy}{|x-y|^{d-1}}.
\endaligned
$$
The rest of the proof is the same as that of Theorem \ref{theorem-3.1.2}.
\end{proof}

\subsection{Lipschitz estimates}

In this subsection we give the proof of the estimate (\ref{Neumann-derivative}).
We also establish an $O(\varep^t)$ estimate for
$u_\varep -u_0 -\big\{ \Psi_{\varep, j}-P_j\big\} \frac{\partial u_0}{\partial x_j}$
in $W^{1,p}(\Omega)$ for any $t\in (0,1)$ and $1<p<\infty$.

\begin{lemma}\label{lemma-4.2.0}
Let $\Omega$ be a bounded $C^{1,\eta}$ domain for some $\eta\in (0,1]$
and
$$
u_\varep (x)=\text{\rm p.v.}
\int_\Omega \frac{\partial^2}{\partial x_i\partial y_j}
\big\{ N_\varep (x,y)\big\} f(y)\, dy
$$
for some $1\le i,j\le d$.
Then
\begin{equation}\label{estimate-4.2.0}
\|u_\varep\|_{L^\infty(\Omega)}
\le C\big\{ \ln [\varep^{-1}M +2] + M^\eta\big\} \| f\|_{L^\infty(\Omega)}
+ C \varep^\eta H_{\varep, \eta} (f),
\end{equation}
where $M=\text{\rm diam} (\Omega)$ and
$H_{\varep, \eta} (f) =\sup \big\{ \frac{|f(x)-f(y)|}{|x-y|^\eta}: x,y\in \Omega
\text{ and } |x-y|<\varep\big\}$.
\end{lemma}

\begin{proof}
For $x\in \Omega$, choose $\hat{x}\in \partial\Omega$
such that $|x-\hat{x}|=\text{\rm dist}(x, \partial\Omega)$.
Note that
$$
\aligned
u_\varep (x)
= &\int_\Omega 
\frac{\partial^2}{\partial x_i\partial y_j}
\big\{ N_\varep (x,y)\big\}  \cdot \{  f(y)- f(x)\} \, dy\\
 &\quad +
\int_{\partial\Omega}
\big\{ n_j(y)-n_j (\hat{x})\big\} \cdot \frac{\partial}{\partial x_i}
\big\{ N_\varep (x,y) \big\} \cdot f(x)\, d\sigma (y),
\endaligned
$$
where we have used the fact $\int_{\partial\Omega} N_\varep (x,y)\, d\sigma (y)=0$.
This, together with the estimates  $|\nabla_x N_\varep (x,y)|\le C |x-y|^{1-d}$
and $|\nabla_x\nabla_y N_\varep (x,y)|\le C |x-y|^{-d}$, gives
$$
\aligned
\| u_\varep \|_{L^\infty(\Omega)}
&\le C \int_\Omega \frac{|f(y)-f(x)|}{|x-y|^d}\, dy
+C \| f\|_{L^\infty(\Omega)}
\int_{\partial\Omega} \frac{d\sigma(y)}{|y-\hat{x}|^{d-1-\eta}}\\
&\le C\| f\|_{L^\infty(\Omega)}
\int_{\Omega\setminus B(x, \varep)} \frac{dy}{|x-y|^d}
+CH_{\varep, \eta} (f) \int_{|y-x|<\varep} 
\frac{dy}{|y-x|^{d-\eta}}
+ C\| f\|_{L^\infty(\Omega)} M^\eta\\
&\le C \| f\|_{L^\infty(\Omega)} \ln [\varep^{-1}M+2]
+C \varep^\eta H_{\varep, \eta} (f) 
+ C\| f\|_{L^\infty(\Omega)} M^\eta.
\endaligned
$$
This completes the proof.
\end{proof}

The following lemma provides a Lipschitz estimate
for local solutions with Neumann boundary conditions.

\begin{lemma}\label{lemma-4.2.1}
Let $\Omega$ be a bounded $C^{2,\eta}$ domain for some $\eta\in (0,1)$.
Suppose that $u_\varep \in H^1(D_{3r})$, $u_0\in C^{2, \eta}(D_{3r})$ and
$\mathcal{L}_\varep (u_\varep) =\mathcal{L}_0(u_0)$ in $D_{3r}$.
Also assume that $\frac{\partial u_\varep}{\partial \nu_\varep}
=\frac{\partial u_0}{\partial \nu_0}$ on $\Delta_{3r}$.
Then, if $0<\varep<(r/2)$,
\begin{equation}
\aligned
& \|\frac{\partial u_\varep^\alpha}{\partial x_i}
-\frac{\partial}{\partial x_i}
\big\{ \Psi_{\varep, j}^{\alpha\beta}\big\} 
\cdot \frac{\partial u_0^\beta}{\partial x_j}\|_{L^\infty(D_r)}\\
&\le \frac{C}{r^{d+1}}
\int_{D_{3r}} |u_\varep -u_0|\, dx
+C\varep r^{-1}\ln [\varep^{-1}M +2] \|\nabla u_0\|_{L^\infty(D_{3r})}\\
&\qquad +C\varep^{1-\rho} r^\rho \ln [\varep^{-1}M+2] \|\nabla^2 u_0\|_{L^\infty(D_{3r})}
+C\varep r^\rho \ln [\varep^{-1}M +2] \|\nabla^2 u_0\|_{C^{0, \rho}(D_{3r})}
\endaligned 
\end{equation}
for any $0<\rho<\min (\eta, \tau)$,
where $(\Psi_{\varep, j}^\beta)$ denotes the matrix of Neumann correctors for
$\mathcal{L}_\varep$ in $\Omega$ and $M=\text{\rm diam}(\Omega)$.
\end{lemma}

\begin{proof}
By rescaling and translation we may assume that $r=1$ and $0\in D_1$.
Let 
$$
w_\varep =u_\varep (x)-u_0(x)- \big\{ \Psi_{\varep, j}^\beta -P_j^\beta\big\} \cdot 
\frac{\partial u_0^\beta}{\partial x_j}.
$$ 
Choose a $C^{2, \eta}$ domain such that $D_{2}\subset \wD\subset D_{3}$.
 We now write
$$
w_\varep (x)
=\int_{\widetilde{D}} \wN_\varep (x,y) \mathcal{L}_\varep (w_\varep)\, dy
+\int_{\partial\widetilde{D}} \wN_\varep (x,y) \frac{\partial w_\varep}{\partial\nu_\varep}\, 
d\sigma(y)
+\frac{1}{|\partial\widetilde{D}|}\int_{\partial\widetilde{D}} w_\varep
$$
for $x\in D_2$,
where $\wN_\varep (x,y)$ denotes the matrix of Neumann functions
 for $\mathcal{L}_\varep$ in $\wD$.
In view of Propositions \ref{prop-2.1} and \ref{prop-2.4}, we have
$w_\varep =w_\varep^{(1)} +w_\varep^{(2)}+c$, 
where $c=\frac{1}{|\partial\widetilde{D}|}\int_{\partial\widetilde{D}} w_\varep$,
\begin{equation}\label{4.2.10}
\aligned
w_\varep^{(1)} (x)
=& -\varep \int_{\wD} \frac{\partial}{\partial y_i}
\big\{ \wN_\varep (x,y)\big\} \cdot
\big\{ F_{jik}(y/\varep)\big\} \cdot
\frac{\partial^2 u_0}{\partial y_j\partial y_k}\, dy\\
&-\int_{\wD} \frac{\partial}{\partial y_i}
\big\{ \wN_\varep (x,y)\big\} \cdot
a_{ij}(y/\varep) \big\{ \Psi_{\varep, k}(y) -P_k(y) \big\}
\cdot \frac{\partial^2 u_0}{\partial y_j\partial y_k}\, dy\\
&+\int_{\wD} \wN_{\varep} (x,y) \cdot a_{ij}(y/\varep)
\frac{\partial}{\partial y_j} \big\{ \Psi_{\varep, k}(y) -P_k (y)
 -\varep \chi_k(y/\varep)\big\}
\cdot \frac{\partial^2 u_0}{\partial y_i\partial y_k}\, dy
\endaligned
\end{equation}
and
\begin{equation}\label{4.2.11}
\aligned
w^{(2)}_\varep (x)= &
\varep \int_{\partial\wD}
\wN_\varep (x,y) \cdot n_i (y)  F_{jik} (y/\varep) \cdot 
\frac{\partial^2 u_0}{\partial y_j\partial y_k}\, d\sigma(y)\\
&
\qquad +\int_{\partial\wD}\wN_\varep (x,y) \cdot \left\{ \frac{\partial u_\varep}
{\partial \nu_\varep} -\frac{\partial u_0}{\partial\nu_0}\right\}\, d\sigma(y)
\endaligned
\end{equation}
(we have supressed all superscripts for notational simplicity).

To estimate $w_\varep^{(2)}$ in $D_1$, we note that
$\mathcal{L}_\varep (w_\varep^{(2)})=0$ in $\wD$ and
$$
\frac{\partial} {\partial \nu_\varep} \big\{ w_\varep^{(2)}\big\}
=\varep n_i F_{jik} (x/\varep) \frac{\partial^2 u_0}{\partial x_j\partial x_k}
-\frac{\varep}{|\partial \wD|}\int_{\partial\wD}
n_i F_{jik}(x/\varep) \frac{\partial^2 u_0}{\partial x_j\partial x_k}\, d\sigma
\qquad \text{ on } \Delta_{2},
$$
where we have used $\frac{\partial u_\varep}{\partial\nu_\varep}
=\frac{\partial u_0}{\partial\nu_0}$ on $\Delta_{2}$.
Since $\|\frac{\partial} {\partial \nu_\varep} 
\big\{ w_\varep^{(2)}\big\}\|_{L^\infty(\Delta_{2})}
\le C \varep \| \nabla^2 u_0\|_{L^\infty(D_{3})}$ and
$$
\| \frac{\partial} {\partial \nu_\varep} \big\{ w_\varep^{(2)}\big\}
\|_{C^{0,\rho} (\Delta_{2})}
\le C \varep^{1-\rho} \|\nabla^2 u_0\|_{L^\infty(D_{2})}
+C \varep \|\nabla^2 u_0\|_{C^{0, \rho}(D_{2})},
$$
it follows from the boundary Lipschitz estimate in \cite[Theorem 7.1]{KLS1} that
$$
\|\nabla w_\varep^{(2)}\|_{L^\infty(D_1)}
\le C \varep^{1-\rho} \|\nabla^2 u_0\|_{L^\infty(D_2)}
+C \varep \|\nabla^2 u_0\|_{C^{0, \rho}(D_2)}
+C\int_{D_2} |w_\varep^{(2)}-c|
$$
for any constant $c$. This leads to
$$
\aligned
\|\nabla w_\varep \|_{L^\infty(D_1)}
&\le \|\nabla w_\varep^{(1)}\|_{L^\infty(D_1)}
+ \|\nabla w_\varep^{(2)}\|_{L^\infty(D_1)}\\
& \le C\int_{D_2} |w_\varep|\, dx
+C \|\nabla w_\varep^{(1)} \|_{L^\infty(D_2)}
+C \varep^{1-\rho} \|\nabla^2 u_0\|_{L^\infty(D_3)}\\
&\qquad\qquad
+C\varep \|\nabla^2 u_0\|_{C^{0,\rho}(D_3)}.
\endaligned
$$
Since $|\Psi_{\varep, j}^\beta -P_j^\beta|\le C \varep \ln [\varep^{-1}M +2]$ by Proposition 
\ref{prop-2.4},
we obtain
\begin{equation}\label{4.2.11-5}
\aligned
& \|\frac{\partial u_\varep^\alpha}{\partial x_i}
-\frac{\partial}{\partial x_i}
\big\{ \Psi_{\varep, j}^{\alpha\beta}\big\} 
\cdot \frac{\partial u_0^\beta}{\partial x_j}\|_{L^\infty(D_1)}\\
&\le C \int_{D_2} |u_\varep -u_0|
+C \varep \ln [\varep^{-1}M+2] \| \nabla u_0\|_{L^\infty(D_3)}
+C \varep \ln [\varep^{-1}M+2] \| \nabla^2 u_0\|_{L^\infty(D_3)}\\
& \qquad
+C \varep^{1-\rho} \|\nabla^2 u_0\|_{L^\infty(D_3)}
+C\varep \|\nabla^2 u_0\|_{C^{0,\rho}(D_3)}
+ C\|\nabla w_\varep^{(1)}\|_{L^\infty(D_2)}.
\endaligned
\end{equation}

It remains to estimate $\nabla w_\varep^{(1)}$ on $D_2$.
The first two integrals in the right hand side of (\ref{4.2.10}) may be handled
by applying Lemma \ref{lemma-4.2.0} on $\wD$. Indeed, let
$$
f(x)=-\varep F_{jik} (x/\varep)\cdot \frac{\partial^2 u_0}{\partial x_j\partial x_k}
-a_{ij}(x/\varep)
\big\{ \Psi_{\varep, k}(x) -P_k (x)\big\}
\frac{\partial^2 u_0}{\partial x_j\partial x_k}.
$$
Note that $\| f\|_{L^\infty(\wD)}
\le C \varep \ln [\varep^{-1} M +2] \|\nabla^2 u_0\|_{L^\infty(D_3)}$ and
$$
H_{\varep, \rho} (f)
\le
C \varep^{1-\rho} \ln [\varep^{-1} M+2] \|\nabla^2 u\|_{L^\infty(D_3)}
+C \varep \ln [\varep^{-1}M+2] \|\nabla^2 u_0\|_{C^{0, \rho}(D_3)}.
$$
It follows by Lemma \ref{lemma-4.2.0} that the first two integrals
in the right hand side of (\ref{4.2.10}) are bounded by
$$
C \varep \ln [\varep^{-1}M +2]
\big\{ \varep^{-\rho} \| \nabla^2 u_0\|_{L^\infty(D_3)}
+\|\nabla^2 u_0\|_{C^{0, \rho}(D_3)}\big\}.
$$

Finally, the third integral in (\ref{4.2.10}) is bounded by
\begin{equation}\label{4.2.12}
C \|\nabla^2 u_0\|_{L^\infty(D_3)}
\int_{\wD} \frac{|\nabla_y \big\{ \Psi_{\varep, k} (y)-P_y (y) -\varep \chi_k (y/\varep)\big\}|}
{|x-y|^{d-1}}\, dy.
\end{equation}
Using that $|\nabla_y \{ \Psi_{\varep, k} (y) -P_k (y) -\varep \chi_k (y/\varep)\}|
\le C \min \big(1, \varep [\text{dist}(y, \partial\Omega)]^{-1} \big)$,
one may show that the integral in (\ref{4.2.12}) is bounded by
$C \varep \big[ \ln (\varep^{-1} +2)\big]^2$.
Thus, we have proved that
$$
\|w_\varep^{(1)}\|_{L^\infty(D_2)}
\le C \varep^{1-\rho} \ln [\varep^{-1}M+2] \|\nabla^2 u_0\|_{L^\infty(D_3)}
+C \varep \ln [\varep^{-1} M +2] \|\nabla^2 u_0\|_{C^{0, \rho}(D_3)}.
$$
This, together with (\ref{4.2.11-5}), yields the desired estimate.
\end{proof}

We are now in a position to give the proof of estimate
(\ref{Neumann-derivative}).

\begin{thm}\label{theorem-4.2.3}
Suppose that $A(y)$ satisfies conditions (\ref{ellipticity}),
(\ref{periodicity}) and (\ref{smoothness}).
Also assume that $A^*=A$.
Let $\Omega$ be a bounded $C^{2,\eta}$ domain for some $\eta\in (0,1)$.
Then
\begin{equation}\label{estimate-4.3.2}
\big| \frac{\partial }{\partial x_i} \left\{ N^{\alpha\beta}_\varep (x,y)\right\}
-\frac{\partial }{\partial x_i}
\big\{ \Psi_{\varep, j}^{\alpha\gamma} (x)\big\}
\cdot \frac{\partial}{\partial x_j}
\left\{ N_0^{\gamma\beta}(x,y)\right\} \big|
\le \frac{C\varep^{1-\rho}\ln [\varep^{-1}M +2]}
{|x-y|^{d-\rho}}
\end{equation}
for any $x,y\in \Omega$ and $\rho\in (0,1)$,
where $C$ depends only on $d$, $m$, $\mu$, $\lambda$, $\tau$, $\rho$
and $\Omega$.
\end{thm}

\begin{proof}
Since $|\nabla_x N_\varep (x,y)|\le C |x-y|^{1-d}$ and $|\nabla\Psi_{\varep, j}^\beta|\le C$,
we may assume that $\varep<|x-y|$ and $\rho$ is small.
Fix $x_0, y_0\in \Omega$, $1\le \gamma \le d$ and let $r=|x_0-y_0|/8$.
Let $u^\alpha_\varep(x)= N^{\alpha\gamma}_\varep (x,y_0)$ and $u_0^{\alpha} (x)
=N_0^{\alpha\gamma} (x,y_0)$.
Observe that 
$$
\left\{
\aligned
\mathcal{L}_\varep (u_\varep) & =\mathcal{L}_0 (u_0) =0 \quad \text{ in } D(x_0, r),\\
\left( \frac{\partial u_\varep}{\partial\nu_\varep}\right)^\alpha
&=\left(\frac{\partial u_0}{\partial\nu_0}\right)^\alpha
=-|\partial\Omega|^{-1}\delta^{\alpha\gamma}\quad \text{ on } \Delta(x_0,r).
\endaligned
\right.
$$ 
Also, note that $\|\nabla u_0\|_{L^\infty(D(x_0, r))} \le Cr^{1-d}$, 
$\|\nabla^2 u_0\|_{L^\infty(D(x_0,r))}\le C r^{-d}$
and $\| \nabla^2 u_0\|_{C^{0, \rho}(D(x_0,r))} \le C r^{-d-\rho}$.
Furthermore, it follows from Theorem \ref{theorem-4.1.2} that
$$
\|u_\varep -u_0\|_{L^\infty(D(x_0,r))} \le C \varep r^{1-d} \ln [\varep^{-1} r+2].
$$
Thus, by Lemma \ref{lemma-4.2.1}, we obtain
$$
\|\frac{\partial u_\varep^\alpha}{\partial x_i}
-\frac{\partial}{\partial x_i} \big\{ \Psi_{\varep, j}^{\alpha\beta}\big\} \cdot
\frac{\partial u_0^\beta}{\partial x_j}\|_{L^\infty (D(x_0, r/3))}
\le C \varep^{1-\rho} r^{\rho-d} \ln [\varep^{-1} M+2].
$$
This completes the proof.
\end{proof}

\begin{remark}\label{remark-4.2.1}
{\rm
Under the symmetry condition $A^*=A$, one has 
$N_\varep^{\alpha\beta} (x,y)=N_\varep^{\beta\alpha}(y,x)$.
Thus, it follows from (\ref{estimate-4.3.2}) that for any $x,y\in \Omega$,
\begin{equation}\label{estimate-4.2.3-1}
\big| \frac{\partial }{\partial y_j} \left\{ N^{\alpha\beta}_\varep (x,y)\right\}
-\frac{\partial }{\partial y_j}
\big\{ \Psi_{\varep, \ell}^{\beta \sigma} (y)\big\}
\cdot \frac{\partial}{\partial y_\ell}
\left\{ N_0^{\alpha \sigma}(x,y)\right\} \big|
\le \frac{C\varep^{1-\rho}\ln [\varep^{-1}M +2]}
{|x-y|^{d-\rho}}.
\end{equation}
Fix $\beta$ and $j$.
Let 
$$
u^\alpha_\varep (x) =\frac{\partial }{\partial y_j} \left\{ N^{\alpha\beta}_\varep (x,y)\right\}
\text{ and }
u_0^\alpha (x)
=\frac{\partial }{\partial y_j}
\big\{ \Psi_{\varep, k}^{\beta \sigma} (y)\big\}
\cdot \frac{\partial}{\partial y_k}
\left\{ N_0^{\alpha \sigma}(x,y)\right\}.
$$
Note that $\mathcal{L}_\varep (u_\varep)=\mathcal{L}_0 (u_0)=0
$ in $\Omega\setminus \{ y\}$
and $\frac{\partial u_\varep}{\partial\nu_\varep}
=\frac{\partial u_0}{\partial\nu_0}=0$ on $\partial\Omega$.
We may use Lemma \ref{lemma-4.2.1} and estimate (\ref{estimate-4.2.3-1})
to deduce that if $\Omega$ is $C^{3, \eta}$ for some $\eta\in (0,1)$,
\begin{equation}\label{estimate-4.2.3-2}
\aligned
\big|
\frac{\partial^2}{\partial x_i \partial y_j}
\big\{ N_\varep^{\alpha\beta} (x,y)\big\}
& -\frac{\partial}{\partial x_i}
\big\{ \Psi_{\varep, k}^{\alpha\gamma} (x) \big\}
\cdot
\frac{\partial^2 }{\partial x_k \partial y_\ell}
\big\{ N_0^{\gamma \sigma} (x,y) \big\}
\cdot
\frac{\partial}{\partial y_j}
\big\{ \Psi_{\varep, \ell}^{\beta \sigma}(y) \big\}
\big|\\
&\le 
\frac{C_\rho \,\varep^{1-\rho}\ln [\varep^{-1}M+2]}{|x-y|^{d+1-\rho}}
\endaligned
\end{equation}
for any $x,y\in \Omega$ and $\rho\in (0,1)$, 
where $C_\rho$ depends only on $d$, $m$, $\mu$, $\lambda$, $\tau$, $\rho$ and $\Omega$.
}
\end{remark}

As a corollary of Theorem \ref{theorem-4.2.3}, we obtain an $O(\varep^t)$
estimate for any $t\in (0,1)$
in $W^{1,p}(\Omega)$ for solutions with Neumann boundary conditions.

\begin{thm}\label{theorem-4.2.4}
Suppose that $A(y)$ and $\Omega$ satisfy the same conditions as in Theorem
\ref{theorem-4.2.3}. Let $1<p<\infty$.
For $\varep\ge 0$ and $F\in L^p(\Omega)$ with $\int_\Omega F=0$, let
$u_\varep\in W^{1,p}(\Omega)$ be the solution of the Neumann problem:
$\mathcal{L}_\varep (u_\varep)=F$ in $\Omega$, $\frac{\partial u_\varep}
{\partial\nu_\varep}=0$ and $\int_{\partial\Omega} u_\varep =0$.
Then
\begin{equation}\label{estimate-4.2.4}
\| u_\varep -u_0 - \big\{ \Psi^\beta_{\varep, j} -P^\beta_j\big\}
\frac{\partial u_0^\beta}{\partial x_j} \|_{W^{1,p}(\Omega)}
\le C_t\, \varep^t  \| F\|_{L^p(\Omega)}
\end{equation}
for any $t\in (0,1)$, where
$C_t$ depends only on $d$, $m$, $\mu$, $\lambda$, $\tau$, $t$, $p$
and $\Omega$.
\end{thm}

\begin{proof}
Since $\|\Psi_{\varep, j}^\beta-P_k^\beta\|_{L^\infty(\Omega)}
\le C\varep \ln [\varep^{-1}M +2]$ and $\|\nabla u_0\|_{L^p(\Omega)}+
\|\nabla^2 u_0\|_{L^p(\Omega)}\le C \| F\|_{L^p(\Omega)}$, 
in view of Theorem \ref{theorem-4.1.3},
it suffices to prove that
\begin{equation}\label{4.2.40}
\|\frac{\partial u^\alpha_\varep}{\partial x_i}
-\frac{\partial }{\partial x_i} \big\{ \Psi_{\varep, j}^\beta\big\}
\cdot \frac{\partial u_0^\beta}{\partial x_j} \|_{L^p(\Omega)}
\le C_t \, \varep^t \| F\|_{L^p(\Omega)}.
\end{equation}
We will prove that estimate (\ref{4.2.40}) holds
for any $1\le p\le \infty$.
To this end,
note that $u_\varep (x)=\int_\Omega N_\varep (x,y) F(y)\, dy$, 
by Theorem \ref{theorem-4.2.3},
\begin{equation}\label{4.2.41}
\big| 
\frac{\partial u^\alpha_\varep}{\partial x_i}
-\frac{\partial }{\partial x_i} \big\{ \Psi_{\varep, j}^\beta\big\}
\cdot \frac{\partial u_0^\beta}{\partial x_j}
\big|
\le C_\rho\, \varep^{1-\rho} \ln [\varep^{-1}M+2]
\int_\Omega \frac{|F(y)|\, dy}{|x-y|^{d-\rho}}
\end{equation}
for any $\rho\in (0,1)$.
Note that if $\rho<1-t$, then  $\varep^{1-\rho}\ln [\varep^{-1}M+2]
\le C \varep^{t}$. Estimate (\ref{4.2.40}) follows easily from (\ref{4.2.41}).
\end{proof}

\section{Leibniz Rules for the Dirichlet-to-Neumann map}

Let $\Lambda: H^{1/2}(\partial\Omega)
\to H^{-1/2}(\partial\Omega)$ denote the Dirichlet-to-Neumann map associated
with  $\mathcal{L}=-\Delta$. 
It is known that the estimate
$\|\Lambda (f)\|_{L^p(\partial\Omega)}
\le C_p \| f\|_{W^{1,p}(\partial\Omega)}$ holds
for $1<p<\infty$ if $\Omega$ is $C^1$ \cite{Fabes-1978}, 
and for $1<p<2+\delta$ if $\Omega$ is Lipschitz 
\cite{Jerison-Kenig-1980, Verchota-1984,
Dahlberg-Kenig-1987}.
The goal of this section is to prove the Leibniz estimates that
were used in subsection 3.3.
We remark that our basic line of argument extends to the general
case $\mathcal{L}=-\text{div}(A(x)\nabla)$. 
Indeed,
Theorem \ref{theorem-5.2}
continues to hold for 
if $A(x)$ satisfies conditions (\ref{ellipticity}),
$\|\nabla A\|_\infty\le C$ and $A^*=A$.
Under the same conditions on $A$, Theorem \ref{theorem-5.1} also holds in the 
scalar case $(m=1)$. If $m>1$, estimate (\ref{estimate-5.1}) holds
for $C^1$ domains.
The details will be given elsewhere.

\begin{thm}\label{theorem-5.1}
Let $\mathcal{L}=-\Delta$ and $\Omega$ be a bounded Lipschitz domain.
Then
\begin{equation}\label{estimate-5.1}
\|\Lambda (fg)-f \Lambda (g)\|_{L^2(\partial\Omega)}
\le C \| f\|_{H^1(\partial\Omega)} \| g\|_{L^\infty(\partial\Omega)}
\end{equation}
for any $f\in H^1(\partial\Omega)$ and $g\in L^\infty(\partial\Omega)$.
\end{thm}

\begin{proof}
Let $w$ be the solution to the $L^2$ Dirichlet problem with data $fg$; i.e.,
 $\mathcal{L} (w)=0$ in $\Omega$, $w=fg$ on $\partial\Omega$
and $(w)^*\in L^2(\partial\Omega)$, where $(w)^*$ denotes the
nontangential maximal function of $w$.
Let $u$ and $v$ denote the solutions of the $L^2$ Dirichlet problem with 
boundary data $f$ and $g$ respectively.
Then
$$
\Lambda( fg) -f\Lambda(g)-g\Lambda (f)
=\frac{\partial}{\partial\nu} \big\{ w-uv\big\}.
$$
To estimate $\partial (w-uv)/\partial\nu$ in $L^2(\partial\Omega)$,
let $h$ be a function such that 
$\mathcal{L}(h)=0$ in $\Omega$ and $(h)^*\in L^2(\partial\Omega)$.
Since $w-uv=0$ on $\partial\Omega$ and
$$
\mathcal{L} (w-uv)=2\frac{\partial u}{\partial x_j}\frac{\partial v}
{\partial x_j}
=2\frac{\partial}{\partial x_j} \left\{ \frac{\partial u}{\partial x_j}
\cdot v \right\},
$$
by the Green's formula, we obtain
\begin{equation}\label{5.1.1}
\int_{\partial\Omega} 
\frac{\partial}{\partial\nu} \big\{ w-uv\big\} \cdot h
=2\int_\Omega \frac{\partial u}{\partial x_j}
\cdot v\cdot \frac{\partial h}{\partial x_j}
-2\int_{\partial\Omega} \Lambda(f) \cdot g\cdot h.
\end{equation}
Thus, by duality, it suffices to prove the following estimate
\begin{equation}\label{5.1.2}
\big|
\int_\Omega \frac{\partial u}{\partial x_j}
\cdot v\cdot \frac{\partial h}{\partial x_j}
\big|
\le C \| f\|_{H^1(\partial\Omega)} \| g\|_{L^\infty(\partial\Omega)}
\| h\|_{L^2(\partial\Omega)}.
\end{equation}

The trilinear estimate (\ref{5.1.2}) on Lipschitz domains
 is a consequence of 
a bilinear estimate due to B. Dahlberg \cite{Dahlberg-1986}
(also see \cite{Hofmann-2008} for related work).
Indeed, since $\Delta (h)=0$ in $\Omega$ and $(h)^*\in L^2(\partial\Omega)$,
it follows from estimate (1.7) in \cite{Dahlberg-1986} that
the left hand side of (\ref{5.1.2}) is bounded by
\begin{equation}\label{5.1.3}
C \| h\|_{L^2(\partial\Omega)}
\left\{
\int_{\partial\Omega}
|\big((\nabla u) v\big)^*|^2\, d\sigma
+\int_\Omega |\nabla^2 u|^2 |v|^2 \delta (x)\, dx
+\int_\Omega |\nabla u|^2 |\nabla v|^2 \delta (x)\, dx
\right\}^{1/2},
\end{equation}
where $\delta (x) =\text{dist}(x, \partial\Omega)$.
The first two integals in (\ref{5.1.3}) are bounded by
$C \| f\|^2_{H^1(\partial\Omega)} \| g\|^2_{L^\infty(\partial\Omega)}$.
This follows from  estimates $\|(\nabla u)^*\|_{L^2(\partial\Omega)}
\le C \| f\|_{H^1(\partial\Omega)}$ 
and $\int_\Omega |\nabla^2 u|^2 \delta (x)\, dx
\le C \| f\|^2_{H^1 (\partial\Omega)}$ \cite{Jerison-Kenig-1980, Dahlberg-1980} 
as well as
the maximum principle
$\|v\|_{L^\infty(\Omega)} \le \| g\|_{L^\infty(\partial\Omega)}$.

Finally, to handle the last integral in (\ref{5.1.3}), we use the fact
that $|\nabla v (x)|^2 \delta (x) dx$ is a Carleson measure
whose norm is less than $C\|g\|^2_{BMO(\partial\Omega)}$ \cite{Fabes-Neri-1980}.
This implies that the integral is bounded by
$C \|(\nabla u)^*\|^2_{L^2(\partial\Omega)} \|g\|^2_{L^\infty(\partial\Omega)}$.
Estimate (\ref{5.1.2}) now follows.
\end{proof}

If $\Omega$ is smooth, then $\Lambda$ is a pseudo-differential operator of
order one. In this case the $L^p$ boundedness of the
commutator $[\Lambda, x_i]$ in the next theorem is well known.

\begin{thm}\label{theorem-5.2}
Let $\mathcal{L}=-\Delta$ and $\Omega$ be a bounded $C^1$ domain.
Then
\begin{equation}\label{estimate-5.2}
\|\Lambda (f x_i)-x_i \Lambda (f)\|_{L^p(\partial\Omega)}
\le C_p \| f\|_{L^p(\partial\Omega)}
\end{equation}
for any $1<p<\infty$ and $f\in L^p(\partial\Omega)$.
If $\Omega$ is Lipschitz, the estimate (\ref{estimate-5.2}) holds for
$p=2$.
\end{thm}

\begin{proof}
Let $\Omega$ be a bounded $C^1$ domain.
Let $w$ and $u$ be harmonic functions in $\Omega$ 
with boundary data $fx_i$ and $f$ respectively
such that $(w)^*, (u)^* \in L^p(\partial\Omega)$.
Then
$\Lambda(fx_i)-f\Lambda (x_i)-x_i \Lambda (f)
=\partial\{ w-x_i u\}/\partial\nu$.
Let $h$ be a harmonic function in $\Omega$ such that
$(h)^*\in L^{p^\prime}(\partial\Omega)$.
Since $w-x_i u=0$ on $\partial\Omega$ and
$\mathcal{L}(w-x_i u)=2 \frac{\partial u}{\partial x_i}$,
\begin{equation}\label{5.2.1}
\big| \int_{\partial\Omega} \frac{\partial}{\partial\nu}
\big\{ w-x_i u\big\} \cdot h\big|
 =2 \big|\int_\Omega \frac{\partial u}{\partial x_i} \cdot h\big|
\le C \| f\|_{L^p(\partial\Omega)}
\| h\|_{L^{p^\prime}(\partial\Omega)}.
\end{equation}
By duality this gives (\ref{estimate-5.2}).
We remark that the inequality in (\ref{5.2.1})
was proved in \cite{dkv-biharmonic} for $p=2$ in Lipschitz domains.
The same argument there also gives the inequality
for $1<p<\infty$ in $C^1$ domains.
Thus we may conclude that estimate (\ref{estimate-5.2})
holds for $1<p<\infty$ if $\Omega$ is $C^1$, and for $p=2$
if $\Omega$ is Lipschitz.
\end{proof}

\bibliography{kls3}

\small
\noindent\textsc{Department of Mathematics, 
University of Chicago, Chicago, IL 60637}\\
\emph{E-mail address}: \texttt{cek@math.uchicago.edu} \\

\noindent \textsc{Courant Institute of Mathematical Sciences, New York University, New York, NY 10012}\\
\emph{E-mail address}: \texttt{linf@cims.nyu.edu}\\

\noindent\textsc{Department of Mathematics, 
University of Kentucky, Lexington, KY 40506}\\
\emph{E-mail address}: \texttt{zshen2@uky.edu} \\

\noindent \today

\end{document}